\documentclass[11pt,a4paper]{article}
\usepackage[utf8]{inputenc}
\usepackage[T1]{fontenc}

\usepackage[t,lf]{spectral}

\usepackage{amsmath}
\usepackage{amsfonts}
\usepackage{amssymb}
\usepackage{amsthm}
\usepackage{ upgreek }

\usepackage[dvipsnames]{xcolor}

\usepackage{marginnote}

\usepackage{csquotes}

\usepackage{stmaryrd}

\usepackage{enumitem}

\usepackage{mathtools}

\usepackage{hyperref}
\hypersetup{
	pagebackref=true,    
	colorlinks=true,        
	breaklinks=true,       
	urlcolor= blue,        
	linkcolor= blue,        
	citecolor=blue,
}

\usepackage{cleveref}
\usepackage{autonum} 
\usepackage{calc}

\usepackage{pgf,tikz}
\usetikzlibrary{matrix,automata,arrows,positioning,calc,patterns,fit,calc,
decorations.pathreplacing,decorations.pathmorphing,decorations.markings,matrix,
tikzmark,shadows.blur}
\usetikzlibrary{cd}
 \usepgflibrary{fpu}

\usepackage{filecontents}

\crefformat{equation}{(#2#1#3)}

\newtheorem{theorem}{Theorem}[section]
\newtheorem{proposition}[theorem]{Proposition}
\newtheorem{lemma}[theorem]{Lemma}
\newtheorem{corollary}[theorem]{Corollary}

\theoremstyle{definition}

\newtheorem{remarkx}[theorem]{Remark}
\newenvironment{remark}
  {\pushQED{\qed}\remarkx}
  {\popQED\endremarkx}

\numberwithin{equation}{section}

\input{macros.sty}

\usepackage{pgfplots}
\pgfplotsset{compat=newest}
\usepgfplotslibrary{groupplots}
\usetikzlibrary{datavisualization}

\pgfplotsset{
    every axis/.append style={
        label style={font=\small},
        tick label style={font=\small},
        legend style={font=\small},
        title style={font=\small},
    }
}

\usepackage{subfigure}

\colorlet{mygreen}{green!80!black}

\usepackage{caption}
\newcommand{\xf}{x_{\textup{fin}}}
\newcommand{\yf}{y_{\textup{fin}}}
\newcommand{\ybarf}{\bar{y}_{\textup{fin}}}
\newcommand{\ystar}{y^\star}
\newcommand{\xin}{x_{\textup{in}}}
\newcommand{\yin}{y_{\textup{in}}}
\newcommand{\ybarin}{\bar{y}_{\textup{in}}}

\newcommand{\cin}{c_{\textup{in}}}
\newcommand{\rf}{r_{\textup{fin}}}
\newcommand{\Omegahat}{\widehat{\Omega}}

\begin{document}


\title{The dynamic saddle--node bifurcation\\ with noise on the slow variable}
\author{Baptiste Bergeot, Nils Berglund and Israa Zogheib}
\date{December 11, 2025. Revised version, April 21, 2026.}   

\maketitle

\begin{abstract}
In this work, we analyse the effect of adding Gaussian white noise to the slow 
variable of a slow--fast system passing through a saddle--node (or fold)
bifurcation. This problem is mainly motivated by applications to non-equilibrium energy 
sinks. While the effect of adding noise to the fast variable, which is important 
for noise-induced tipping, has been previously analysed in detail, the case where 
the slow variable is perturbed by noise has not been considered before. Our main 
result is that the noise increases the slow variable on average. We compute 
the effect of the noise, to lowest order, on the expectation and variance of the slow 
variable after the bifurcation. The contribution of the noise can be explicitly expressed 
in terms of Airy functions. We also provide numerical simulations, which show that the 
expansion to lowest order matches the observations for fairly large values of the 
noise intensity. 
\end{abstract}

\leftline{\small 2020 {\it Mathematical Subject Classification.\/} 
60H10, 
34F05 (primary), 
70K70, 
70L05 (secondary). 
}
\noindent{\small{\it Keywords and phrases.\/}
Slow--fast system,
saddle--node bifurcation,
fold bifurcation, 
noise,
stochastic differential equation, 
non-linear energy sink.

\section{Introduction}
\label{sec:intro} 


In this work, we are interested in two-dimensional slow--fast systems 
near a saddle--node, or fold bifurcation, in the presence of noise. 
In the deterministic case, the study of these systems goes back to the 
works~\cite{PontRod} and~\cite{Haberman}. If $\eps$ denotes the time scale 
separation, the fast variable is known to track the stable part of the critical manifold 
at a distance that grows up to order $\eps^{1/3}$ when approaching the bifurcation 
point. After a delay of order $\eps^{2/3}$, the fast variable quickly exits the 
neighbourhood of the bifurcation point. 

The case when Gaussian white noise is added to the fast variable has been studied 
in~\cite{BG3} (see also~\cite[Section~3.3]{BGbook}). The main effect in that case is 
the existence of a threshold for the noise intensity, scaling like $\sqrt{\eps}$. 
When the noise intensity is below this threshold, the system behaves with high 
probability like the deterministic system. However, when the noise intensity 
exceeds the threshold, the fast variable is likely to cross the unstable 
critical manifold some time before reaching the bifurcation point. This has 
important consequences on noise-induced tipping~\cite{Lentonetal} and early 
warning signs~\cite{Schefferetal}, and has been further investigated 
in many works, see for 
instance~\cite{Kuehn_2011_tipping,PhysRevE.85.046202,Kuehn_2013_critical,10.1115/1.4034128,PhysRevE.96.030201,Kim:2018aa,Kuehn_Romano18}.

By contrast, in this work we are interested in the case where noise is added to 
the slow variable of the system, which leads to a very different behaviour. 
This study is motivated by the investigation of nonlinear energy sinks 
(NESs)~\cite{Gendelman2001,vakakis2001}, particularly when they are employed to 
mitigate self-sustained oscillations~\cite{Gendelman2010Phys}. In this context, the dynamical system under 
consideration generally consists of a mechanical self-sustained oscillator 
(SSO)--whose oscillations we seek to attenuate--coupled with the NES, i.e., 
a small mass connected to the SSO via a linear damper and an essentially nonlinear 
spring (typically purely cubic). After applying an averaging procedure, the coupled 
system reduces to a (2,1)-fast--slow system. The two fast variables relate to the 
amplitude of the relative displacement between the SSO and the NES, and the phase 
difference between this relative displacement and the displacement of the SSO. 
The slow variable is related to the amplitude of the SSO motion. As presented in 
Figure~\ref{fig:intro1}, the deterministic dynamics of this fast--slow system are 
primilary governed (i) by its $S$-shaped one-dimensional critical manifold and 
(ii) by its equilibria, whose stability and position on the critical manifold 
depend on a bifurcation parameter. In particular, certain values of this parameter 
result in a configuration supporting relaxation oscillations with amplitudes smaller 
than those of the self-sustained oscillations that would occur without the NES 
(see Figure~\ref{fig:intro1}(a)). For larger values of the parameter, these relaxation 
oscillations are no longer possible, and a stable equilibrium is reached, corresponding 
to an SSO amplitude close to that observed in the absence of the NES 
(see Figure~\ref{fig:intro1}(b)). It therefore appears that the position of the 
trajectory’s endpoint on the right-hand attracting branch of the critical manifold, 
together with the location of the rightmost unstable equilibrium on this branch, 
determines which situation occurs: if the endpoint lies below the equilibrium, 
mitigation occurs; if it lies above, there is no mitigation. This means that, 
for a certain value of the bifurcation parameter, a tipping occurs between the 
two situations. In the seminal works of Gendelman 
\textit{et al.}~\cite{Gendelman2010SIAM,Gendelman2010Phys}, 
a multiple time scales analysis~\cite{nayfeh2008perturbation} is used to 
estimate the tipping value of the parameter by (i) obtaining analytically 
the equilibria and (ii) approximating the endpoint’s ordinate with that of 
the critical manifold’s left fold point, where the manifold loses stability 
via a saddle--node bifurcation. This approximation, sometimes referred to as 
the zero-order approximation, does not account for the specific dynamics 
that occur near the fold points, where the normal hyperbolicity of the critical 
manifold is lost. By means of the center manifold theorem, 
Bergeot~\cite{BERGEOT2021116109} addressed this limitation by reducing the 
dynamics near the left fold point to the normal form of a dynamic saddle-node 
bifurcation, which can be solved analytically. This approach allowed the author 
to provide an improved theoretical estimate of the parameter’s tipping value.

\begin{figure}[t!]
\begin{center}
\input{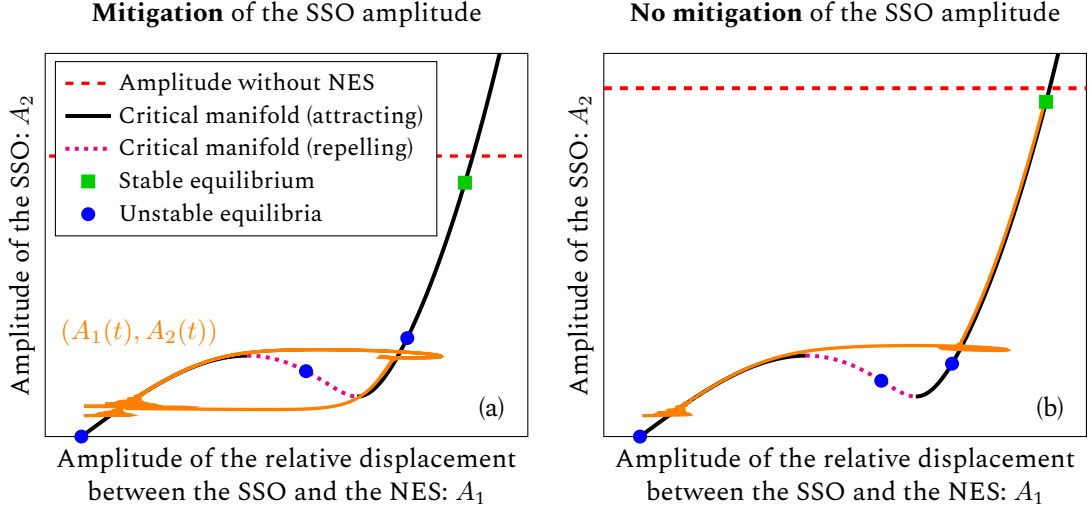}
\end{center}
\caption[]{Behaviour of the deterministic SSO–NES system in the vicinity of the 
tipping toward unmitigated responses. (a) Parameter below the tipping value 
(mitigated response): the system exhibits relaxation oscillations whose amplitudes 
remain smaller than those of the self-sustained oscillations that would occur 
without the NES. (b) Parameter above the tipping value (unmitigated response): 
the trajectory converges to a stable equilibrium, corresponding to an SSO amplitude 
close to that observed in the absence of the NES.}
\label{fig:intro1} 
\end{figure}

The influence of noise on this class of systems was investigated 
in \cite{BERGEOT2023104351}, where a small-amplitude white-noise forcing was applied 
to the SSO. Using stochastic averaging--following Roberts and 
Spanos \cite{Roberts1986}--the study first demonstrated that the SSO-NES stochastic 
system can be accurately reduced to a (2,1)-fast--slow system with noise acting only 
the slow variable, the drift part being approximated by the corresponding 
deterministic system. A Monte Carlo simulation approach then showed that the 
stochastic forcing can significantly alter the dynamic behaviour of the corresponding 
deterministic system, in particular by rendering certain configurations dangerous 
(i.e., unmitigated) that were safe in the deterministic case. The results also reveal 
complex stochastic dynamics, showing that the reasoning used to predict the system’s 
deterministic behaviour can fail in the presence of noise. The present work paves 
the way for identifying the underlying mechanisms that govern the stochastic behaviour. 
To do this, we exploit the fact that, near the left fold point of the deterministic 
critical manifold, the SSO-NES stochastic system reduces to the normal form of a 
dynamic saddle–node bifurcation (as done for the deterministic system 
in~\cite{BERGEOT2021116109}), with noise acting on the slow variable.

Our main result, Theorem~\ref{thm:main}, provides the first relevant terms of an expansion 
in the noise intensity of the first two moments of the slow variable after passing through the 
bifurcation. The most notable feature is that the expectation of the slow variable is 
increased by the noise.
In terms of the SSO-NES system, this suggests that noise may slightly 
promote the unmitigated responses.
While the effect is relatively small for weak noise, since the expectation 
changes by an amount proportional to the variance of the noise, numerical simulations indicate 
that our expansion is accurate for fairly large values of that variance. 

The remainder of this article is structured as follows. In Section~\ref{sec:results}, we 
give a precise formulation of the mathematical set-up, state and discuss the main results, and present 
numerical simulations.
The remaining sections are devoted to the proofs of the main results. 
Section~\ref{sec:structure} gives a brief overview of the structure of the proofs. 
Section~\ref{sec:one_slice} provides a detailed analysis in a small interval of the fast 
variable. In Section~\ref{sec:slices}, the small intervals are stitched together, to provide 
the proof of the main theorem on moments of the slow variable. 
Section~\ref{sec:DV} contains the proof of an auxiliary result, giving explicit expressions 
for the moments in terms of Airy functions.
Appendix~\ref{sec:airy} gives a short summary of some useful properties of Airy functions.

\subsubsection*{Acknowledgements:}

The authors thank two anonymous reviewers for their careful reading of the manuscript, 
and for their constructive remarks, that led to improvements in the presentation.


\section{Results}
\label{sec:results} 


\subsection{Set-up}
\label{ssec:setup} 

We are interested in the slow--fast system 
\begin{align}
 \6\bar x_{\bar t} &= \frac{1}{\eps} (\bar y_{\bar t} + \bar x_{\bar t}^2) \6\bar t \\
 \6\bar y_{\bar t} &= \6\bar t + \bar\sigma \6\overbar W_{\bar t}\;,
 \label{eq:sn_not_scaled} 
\end{align}
describing the normal form at a saddle-node bifurcation point. Here $\eps > 0$ 
is a small parameter measuring time scale separation, and 
$(\overbar W_{\bar t})_{\bar t\geqs0}$ is a standard Wiener process describing white noise, 
while $\bar\sigma > 0$ is a small parameter measuring the noise intensity. 

The scaling
\begin{equation}
 \bar t = \eps^{2/3} t\;, \qquad 
 \bar x_{\bar t} = \eps^{1/3} x_t\;, \qquad 
 \bar y_{\bar t} = \eps^{2/3} y_t 
 \label{eq:scaling_eps} 
\end{equation} 
results in the system
\begin{align}
 \6 x_t &= (y_t + x_t^2) \6t \\
 \6 y_t &= \6t + \sigma\6W_t\;,
 \label{eq:sn_scaled} 
\end{align}
where $(W_t)_{t\geqs0}$ is again a standard Wiener process, owing to the scaling 
property of Brownian motion, and 
\begin{equation}
 \sigma = \eps^{1/3}\bar \sigma
\end{equation} 
is a rescaled noise intensity. The scaling~\eqref{eq:scaling_eps} is 
motivated by the fact that the new variables $x$ and $y$ change by comparable 
amounts in a given times interval, when they have order $1$.
We will keep in mind the fact that results obtained for 
the rescaled system~\eqref{eq:sn_scaled} on a rectangle $[-a,a]\times[-b,b]$ will 
translate into results for the original system~\eqref{eq:sn_not_scaled} on 
the scaled rectangle $[-\eps^{1/3}a,\eps^{1/3}a]\times[-\eps^{2/3}b,\eps^{2/3}b]$. 
It is therefore of interest to obtain results for possibly large values of $a$ and $b$. 

The deterministic case $\sigma = 0$ is well-known~\cite{PontRod,Haberman}. 
The slow--fast system~\eqref{eq:sn_not_scaled} 
has a critical manifold $\setsuch{(\bar x,\bar y)}{\bar y = -\bar x^2}$, separating 
orbits with increasing and decreasing $\bar x$. After scaling, this manifold takes the equivalent 
form $\setsuch{(x,y)}{y = -x^2}$.

In the deterministic case $\sigma = 0$, using the theory of Riccati equations, the general 
solution of~\eqref{eq:sn_scaled} with initial condition $(\xin,\yin)$ can be written 
\begin{align}
 x^{\det}(t) &= \frac{\Ai'(-\yin-t) + K \Bi'(-\yin-t)}{\Ai(-\yin-t) + K\Bi(-\yin-t)}\;,\\
 y^{\det}(t) &= \yin + t\;, 
 \label{eq:xdet} 
\end{align} 
where $\Ai$ and $\Bi$ are Airy function (cf.\ Appendix~\ref{sec:airy}), 
and $K$ is determined by the condition $x^{\det}(0) = \xin$. It will sometimes be convenient to write, 
with a slight abuse of notation,  
\begin{equation}
\label{eq:xdet_y} 
 x^{\det}(y) = \frac{\Ai'(-y) + K \Bi'(-y)}{\Ai(-y) + K\Bi(-y)}
\end{equation} 
for the solution parametrised in terms of the $y$ coordinate. 
We will be particularly interested in the case $K = 0$, where we have the particular solution 
(called \emph{slow solution} in the case of the unscaled system~\eqref{eq:sn_not_scaled})
\begin{align}
 x^{\det}_0(t) &= \frac{\Ai'(-\yin-t)}{\Ai(-\yin-t)},\\
 y^{\det}_0(t) &= \yin + t\;.
 \label{eq:x0} 
\end{align}
This is because the asymptotics~\eqref{eq:Airy_asymptotics} of Airy functions imply that 
$(x^{\det}_0(t),y^{\det}_0(t))$ 
converges to the critical manifold as $t\to-\infty$. On the other hand, we have 
\begin{equation}
\lim_{t\to t^\star} x^{\det}_0(t) = +\infty\;,
\qquad 
\lim_{t\to t^\star} y^{\det}_0(t) = y^\star\;,
\end{equation} 
where $t^\star = \ystar - \yin$, and 
\begin{equation}
 \ystar = 2.338107410459767 \dots
 \label{eq:ystar}
\end{equation} 
is the negative of the largest zero of the Airy function $\Ai$ (cf.\ the Online 
Encyclopedia of Integer Sequences \texttt{OEIS A096714}). 

Figure~\ref{fig:dsnb} shows an example of the general solution \eqref{eq:xdet} 
of~\eqref{eq:sn_scaled} with $\sigma = 0$, i.e., the parametric curve $\left(x^{\det}(t), y^{\det}(t)\right)$ 
for the initial condition $(\xin,\yin)=(-3,-2)$. The figure also displays the reference 
trajectory $\left(x_0^{\det}(t), y_0^{\det}(t)\right)$ (see \eqref{eq:x0}), the critical 
manifold $y=-x^2$, and the limit value~$\ystar$.

\begin{figure}[t!]
\begin{center}
\input{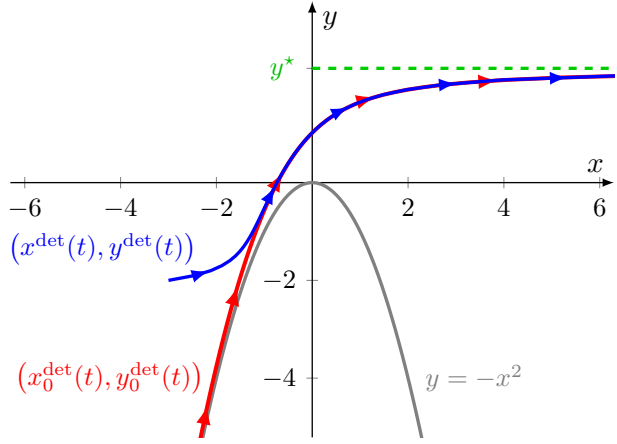}
\end{center}
\caption{Dynamic saddle-node bifurcation. 
An example of the general solution $\left(x^{\det}(t), y^{\det}(t)\right)$ of~\eqref{eq:sn_scaled} 
(see \eqref{eq:xdet}) for the initial condition $(\xin,\yin)=(-3,-2)$, the reference trajectory 
$\left(x_0^{\det}(t), y_0^{\det}(t)\right)$ (see \eqref{eq:x0}), the critical manifold $y=-x^2$, 
and the limit value $\ystar$ (see \eqref{eq:ystar}).}
\label{fig:dsnb} 
\end{figure}


\subsection{Main result}
\label{ssec:main} 

To formulate our main result, we fix an initial value $(\xin,\yin)$ with 
$\xin < 0$ and $\xin^2 + \yin > 0$, as well as a final value $\xf > 0$. 
Given the solution $(x_t,y_t)$ of~\eqref{eq:sn_scaled} with this initial 
condition, we denote by 
\begin{equation}
\label{eq:def_tau} 
 \tau = \inf\bigsetsuch{t > 0}{x_t = \xf}
\end{equation} 
the first-hitting time of $\xf$. We are interested in properties of the random 
variable $y_\tau$. We will also need the deterministic analogue of~\eqref{eq:def_tau}, 
given by 
\begin{equation}
\label{eq:def_T} 
 T(\xin,\yin;\xf) = \inf\bigsetsuch{t > 0}{x^{\det}(t) = \xf}\;,
\end{equation} 
where $(x^{\det}(t),y^{\det}(t))$ denotes the deterministic solution~\eqref{eq:xdet} 
of the system~\eqref{eq:sn_scaled} with the same initial condition $(\xin,\yin)$. 
We write 
\begin{equation}
 \yf = y^{\det}(T(\xin,\yin;\xf))
 = \yin + T(\xin,\yin;\xf)
\end{equation} 
for the final value of $y^{\det}$. 

\begin{theorem}
\label{thm:main}
Fix a constant $h_0 > 0$. 
Let $\Omega_0$ be the event 
\begin{equation}
\label{eq:def_Omega0} 
 \Omega_0 
 = \biggsetsuch{\omega}{\sup_{0\leqs t\leqs \tau} \bigabs{y_t(\omega) - y^{\det}(t)} \leqs h_0}\;.
\end{equation} 
There exists a constant $\kappa > 0$, depending only on $h_0$, such that 
\begin{equation}
\label{eq:bound_P_Omega0} 
 \fP(\Omega_0^c) \leqs \frac{\yf - \yin}{\sigma} \e^{-\kappa/\sigma}\;.
\end{equation} 
Furthermore, the expectation and variance of $y_\tau\indicator{\Omega_0}$ satisfy  
\begin{align}
\label{eq:bound_expec_main} 
 \bigexpecin{(\xin,\yin)}{y_\tau \indicator{\Omega_0}} 
 &= \yf
 + \frac12 \sigma^2 D(\xin,\yin;\xf) + \Order{\sigma^3}\;, \\
 \Varin{(\xin,\yin)}{y_\tau \indicator{\Omega_0}}
 &= \sigma^2 V(\xin,\yin;\xf) + \Order{\sigma^3}\;. 
\label{eq:bound_variance_main} 
\end{align} 
The functions $D$ and $V$ are given by 
\begin{align}
\label{eq:DV} 
 D(\xin,\yin;\xf)
 &= \int_{\yin}^{\yf} \partial_{yy} T(x^{\det}(y),y;\xf) \6y\;, \\
 V(\xin,\yin;\xf)
 &= \int_{\yin}^{\yf} \bigpar{1 + \partial_y T(x^{\det}(y),y;\xf)}^2 \6y\;, 
 \label{eq:VD}
\end{align}
where we write $\partial_y T$ and $\partial_{yy}T$ for derivatives of the deterministic 
time~\eqref{eq:def_T} with respect to its second argument $\yin$. 
\end{theorem}

The functions $D$ and $V$ can be made more explicit if we assume that the initial condition 
lies on the slow solution~\eqref{eq:x0}, that is, if 
\begin{equation}
\label{eq:xin_slow} 
 \xin = x^{\det}_0(0) = \frac{\Ai'(-\yin)}{\Ai(-\yin)}\;.
\end{equation} 
Note that the slow solution is attracting for $x<0$, so that small changes in $\xin$ will have 
almost no effect on the result.

\begin{proposition}
\label{prop:DV} 
If $\xin$ is given by~\eqref{eq:xin_slow}, then 
\begin{align}
 \lim_{\xf\to\infty} D(\xin,\yin;\xf)
 &= \frac34 + \frac{1}{\Ai'(-\ystar)^2}
 \Biggbrak{2\pi\Biggpar{\frac{\Bi'(-\ystar)}{\Ai'(-\ystar)}\cF(-\yin)-\cG(-\yin)}
 -f(-\yin)}\;,\\
 \lim_{\xf\to\infty} V(\xin,\yin;\xf)
 &= \frac12 \ystar + \frac{\cF(-\yin)}{\Ai'(-\ystar)^4}\;, 
 \label{eq:VDinf}
\end{align}
where 
\begin{align}
 f(z) ={}& \Ai'(z)^2 - z\Ai(z)^2\;, 
\label{eq:def_f} 
 \\
 \cF(z) ={}& \biggpar{\frac{z^3}{2} + \frac18} \Ai(z)^4 
 - \frac{z}{2} \Ai(z)^3 \Ai'(z) 
 - z^2 \Ai(z)^2\Ai'(z)^2 \\
\label{eq:def_F} 
 &{}+ \frac12 \Ai(z)\Ai'(z)^3
 + \frac{z}2 \Ai'(z)^4\;, \\
 \cG(z) ={}&
 \Biggbrak{\biggpar{\frac{z^3}{2} + \frac18}\Ai(z)^3-\frac{3z}8 \Ai(z)^2\Ai'(z)-\frac{z^2}2\Ai(z)\Ai'(z)^2+\frac18\Ai'(z)^3}\Bi(z)\\
\label{eq:def_G} 
&{}+\Biggbrak{-\frac{z}8 \Ai(z)^3-\frac{z^2}2 \Ai(z)^2\Ai'(z) 
+ \frac38\Ai(z)\Ai'(z)^2+\frac{z}2\Ai'(z)^3}\Bi'(z)\;.
\end{align}
In particular, the asymptotics of Airy functions (cf.~\eqref{eq:Airy_asymptotics}) imply 
\begin{align}
 \lim_{\xin\to-\infty} \lim_{\xf\to\infty} D(\xin,\yin;\xf)
 &= \frac34\;,\\
 \lim_{\xin\to-\infty} \lim_{\xf\to\infty} V(\xin,\yin;\xf)
 &= \frac12 \ystar = 1.169053705229883\dots\;. 
 \label{eq:DV_limit} 
\end{align}
Furthermore, 
\begin{equation}
 \lim_{\xf\to\infty} D(\xin,\yin;\xf) > 0
\end{equation} 
for any $\xin\in\R$ (or equivalently, for any $\yin < y^\star$). 
\end{proposition}

Figure~\ref{fig:funprop22} shows graphs of the explicit expressions for the limits of the 
functions $D$ and $V$ as $\xf \to \infty$, as given in equation~\eqref{eq:VDinf}. 
These limits tend to $\frac{3}{4}$ and $\frac12\, y^\star$, respectively, when 
$x_{\mathrm{in}} \to -\infty$, in agreement with equation~\eqref{eq:DV_limit}.

\begin{figure}[t!]
\begin{center}
\input{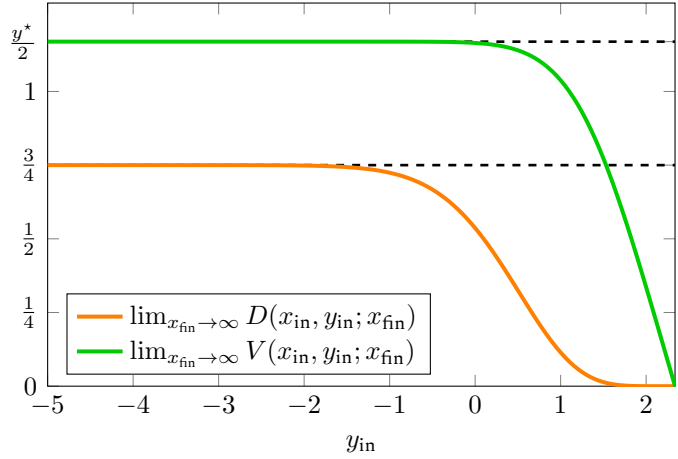}
\end{center}
\vspace{-2mm}
\caption{Plots of the explicit limit expressions of the functions $D$ and $V$ as 
$\xf \to \infty$, derived in equation~\eqref{eq:VDinf}, as a function of $\yin$.}
\label{fig:funprop22} 
\end{figure}

\begin{remark}
We also have explicit expressions for the functions $D$ and $V$ when $\xf$ is finite, 
but they are more complicated, see Proposition~\ref{prop:TxTxx}. 
Figure~\ref{fig:Tderivatives} below shows that the analytical expressions 
for $D$ derived from~\eqref{eq:TyTyy} perfectly matches numerical simulations 
of the integral.
\end{remark}


\subsection{Discussion}
\label{ssec:discussion} 

The most interesting feature of our results is that the expectation of $y_\tau$ 
increases with the noise intensity $\sigma$, since $D$ is positive. Stricly speaking, we 
have proved this only in the limit $\yf \to y^\star$, but numerical simulations indicate 
that $\partial_{yy}T$ is always positive. One caveat of our result is that the expansion 
in powers of $\sigma$ is only guaranteed to be accurate for small $\sigma$. In that case, 
the expected difference, of order $\sigma^2$, is much smaller than the standard deviation 
of $y_\tau$, which has order $\sigma$. Therefore, the probability of having 
$y_\tau > y^{\det}_\tau$ should be roughly $\frac12 + \Order{\sigma}$, which is a small 
effect. This is one of the reasons the analysis is quite technical.

However, in the next section we will show numerical simulations that indicate that the 
expansions~\eqref{eq:bound_expec_main} and~\eqref{eq:bound_variance_main} are actually 
surprisingly accurate, even for $\sigma = 1$. In that case, the expected difference and 
standard deviation are of comparable size, so that the effect of the noise is no longer 
small. 

\begin{remark}
The reason why the expectation of $y_\tau$ increases with the noise intensity 
$\sigma$ is rather subtle. It is the combination of two effects: (1) the fact that the 
expectation increases like the second derivative of the deterministic time 
$T(x^{\det}(y),y;\xf)$, and (2) the convexity of this time as a function of $y$.
Convexity of $T$ is essentially because as $y$ increases, $T$ decreases, while 
remaining positive. As for the expectation increasing like the second 
derivative of $T$, this is related to Jensen's inequality. We present here a simplification 
of the actual argument given in Section~\ref{ssec:expectation}. 
Let $(x_1 = x^{\det}(y_1), y_1)$ and $(x_2 = x^{\det}(y_2), y_2)$ be two 
nearby points on the deterministic reference solution, with $x_1 < x_2$.
We make the Ansatz 
\begin{equation}
 \expecin{x_i,y}{y_\tau} = \yf + \sigma^2 D_i(y)\;,
\end{equation} 
where we recall that $\yf = y + T(x_i,y;\xf)$ for any $y\leqs \yf$, since 
$\dot{y} = 1$ if $\sigma = 0$. 
Therefore, writing $y_2 + \sigma Z$ for the value of $y_t$ when $x_t$ reaches $x_2$,
the strong Markov property implies 
\begin{align}
\label{eq:E_convex} 
 \sigma^2 D_1(y_1)
 &= \bigexpecin{x_1,y_1}{\expecin{x_2,y_2+\sigma Z}{y_\tau}} - \yf \\
 &\simeq \bigexpecin{x_1,y_1}{y_2 + \sigma Z + T(x_2,y_2+\sigma Z;\xf)
 + \sigma^2 D_2(y_2)} - \yf \\
 &= \bigexpecin{x_1,y_1}{T(x_2,y_2+\sigma Z;\xf)} - T(x_2,y_2;\xf) 
 + \sigma \expecin{x_1,y_1}{Z} + \sigma^2 D_2(y_2)\;.
\end{align} 
By Jensen's inequality,
\begin{equation}
 \bigexpecin{x_1,y_1}{T(x_2,y_2+\sigma Z;\xf)} 
 \geqs T(x_2,y_2+\sigma \expecin{x_1,y_1}{Z};\xf)\;.
\end{equation} 
If $\gamma = y_2 - y_1$ is small, one can show that $Z$ has an expectation of 
order $\gamma^2$, but a variance close to $\gamma$. This already shows that if 
we neglect the expectation of $Z$, one has $D_1(y_1) \geqs D_2(y_2)$. 
A more precise computation (see Proposition~\ref{prop:induction_E} for details), 
based on a Taylor-expansion of $T$, shows that 
\begin{equation}
 \bigexpecin{x_1,y_1}{T(x_2,y_2+\sigma Z;\xf)} 
 - T(x_2,y_2;\xf)
 = \sigma^2 \expecin{x_1,y_1}{Z^2}\partial_{yy}T(x_2,y_2;\xf) 
 + \Order{\sigma^2\gamma^2}\;,
\end{equation} 
which implies 
\begin{equation}
 \sigma^2 D_1(y_1) \simeq \sigma^2 D_2(y_2) 
 + \sigma^2 \gamma \partial_{yy}T(x_2,y_2;\xf) + \Order{\sigma^2\gamma^2}\;.
\end{equation} 
Taking $\gamma = \Order{\sigma}$, we can neglect the remainder, and we 
obtain Relation~\eqref{eq:DV} for $D(\xin,\yin;\xf)$ in differential form. 
\end{remark}

One may wonder whether the presence of the indicator function $\indicator{\Omega_0}$ 
is necessary, or only a technical limitation. The answer is that the indicator is indeed 
needed for the mathematical analysis presented here, though it may have little importance
in practice. The main reason is that if a sample path exits the $h_0$-neighbourhood of 
the deterministic solution, it has a larger probability to reach the critical manifold 
$y = -x^2$. In that case, the sample path may move to the left for a while, before crossing 
the critical manifold again and moving to the right. This is a rare event, but it can 
potentially have a large effect on expectations, which we do not take into account
in our analysis. However, simulations shown in the next section indicate that our 
results remain valid without the indicator function. This is because even when sample paths 
move below the critical manifold, leading $x$ to decrease for a while, they are likely 
to behave in the same way as paths arriving from the far left once they have moved 
above the critical manifold again. Therefore, the extra detour made by sample paths 
has almost no influence on the final value of $y_\tau$, though it increases the time
spent near the fold point.

One question of interest is about the effect of undoing the scaling transforming 
the system~\eqref{eq:sn_not_scaled} into the system~\eqref{eq:sn_scaled}. The change 
in expectation now becomes of order $\eps^{4/3}\bar\sigma^2$, while the variance has 
order $\eps^2\bar\sigma^2$. 
One may also wonder whether our results are only valid for $\xin$ and $\xf$ of 
order $1$, or whether they can be extended to $\xin$ and $\xf$ of order $\eps^{-\beta}$
for some $\beta > 0$. Indeed, this would allow to describe the system for $(\bar x,\bar y)$ 
of order $\eps^\alpha$, where $\alpha = \frac13 - \beta$, which would be useful to 
describe the dynamics of a more general system away from the fold point (see also the 
discussion on the centre manifold reduction below).
While we do not provide here a rigorous proof of that fact, we expect this extension 
to be possible, since the quantities $D$ and $V$ vary little 
when $\abs{\xin}$ and $\xf$ are large, and the error terms in 
Proposition~\ref{prop:moments_ytau} actually decrease when $x^2$ increases. 
The main effect is that the prefactor in the probability~\eqref{eq:bound_P_Omega0} 
would scale like $(\ybarf - \ybarin)/(\eps\bar\sigma)$, which is negligible compared to the 
exponential term $\e^{-\kappa/\sigma} = \e^{-\kappa/(\eps^{1/3}\bar\sigma)}$.

For the SSO--NES system motivating this work, our results show that after the 
trajectories pass the left fold point of the critical manifold 
(see Figure~\ref{fig:intro1}), the slow variable -- associated with the amplitude of 
the SSO -- has an expectation larger than the corresponding deterministic solution, 
with the difference becoming larger as the noise intensity increases. This suggests 
that noise may promote unmitigated responses, increasing the probability that 
trajectories end on the right-hand attracting branch above the rightmost unstable 
equilibrium (Figure~\ref{fig:intro1}b). 
This hypothesis will require future verification through numerical and theoretical 
analyses. The numerical results in the following section show that we can be confident 
in the validity of our results even for large values of $\sigma$. However, the centre 
manifold theorem, which allows the reduction to the normal form of a dynamic 
saddle--node bifurcation, is valid only near the left fold point. For the deterministic 
system, this reduction remains relevant beyond the immediate vicinity of the fold, 
particularly concerning the trajectory’s arrival on the right-hand attracting branch of 
the critical manifold (see Figure 5 of \cite{BERGEOT2021116109}), which is crucial 
for determining if mitigation occurs. Whether this extends to the stochastic case 
will be explored in future work, taking into account the stochastic slow dynamics 
after the system reaches the right-hand attracting branch, which also appears to influence 
the system’s tipping, as suggested by results in~\cite{BERGEOT2023104351}.

In more detail, the centre manifold theorem applies at least in a neighbourhood 
of size $\eps^{1/3}$ in the unscaled $\bar x$ variable, and of size $\eps^{2/3}$ 
in the unscaled $\bar y$ variable. This is because higher-order corrections to the 
normal form become, in scaled $(x,y)$-variables, terms of order $\eps^{1/3}$ at most, 
and therefore act as small perturbations to the normal form analysed here. Once trajectories 
leave this neighbourhood, a possible approach is to show that sample paths of 
the stochastic system remain sufficiently close to their deterministic counterpart 
for the effect of the noise to be negligible, compared to its effect near the fold 
point. The reason we expect this to hold is that the duration of the fast phase 
is small compared to the time spent near the fold point (in the spirit 
of~\cite[Theorem~3.1.11]{BGbook}). To obtain optimal results, 
it may be necessary to extend the local analysis to a slightly larger neighbourhood 
of sizes $\eps^\alpha$ in $x$ and $\eps^{2\alpha}$ in $y$, with $\alpha$ slightly less than 
$\frac13$, which decreases the duration of the fast phase spent outside that neighbourhood,
at the price of increasing the size of the error terms due to higher-order corrections 
near the fold point. 


\subsection{Numerical simulations}
\label{ssec:numerics} 

\begin{figure}[t!]
\begin{center}
\input{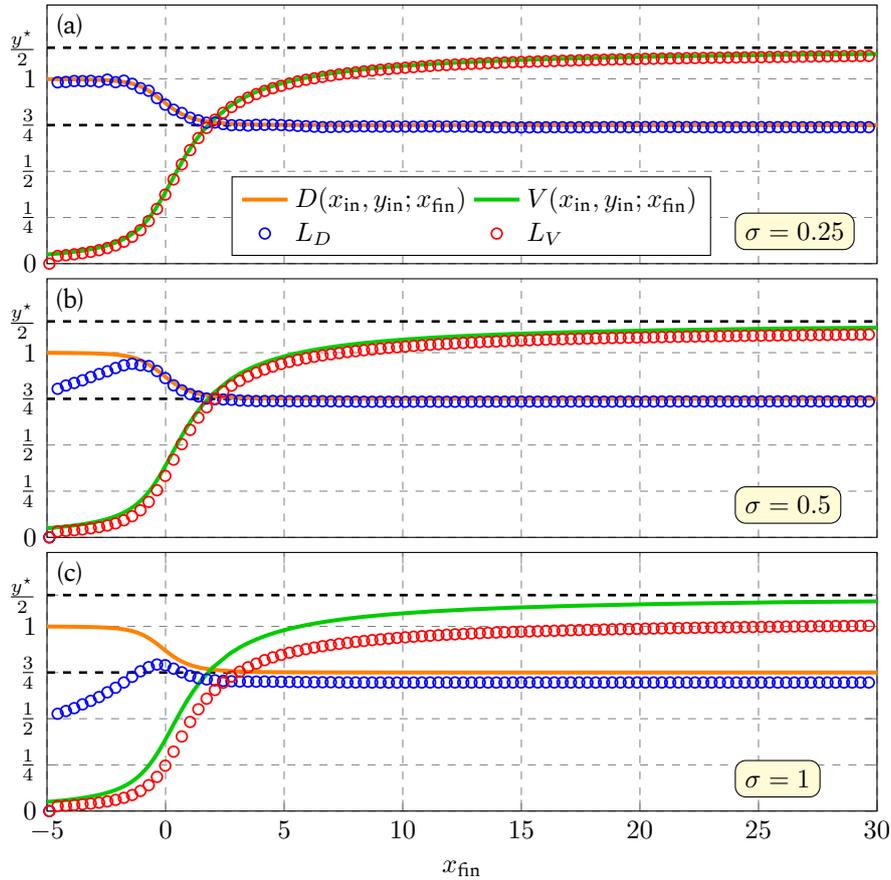}
\end{center}
\caption[]{Comparison between the results of Theorem~\ref{thm:main} and numerical simulations. 
Numerically estimated statistics 
$L_D$ and $L_V$ (see equation \eqref{eq:LDLV}) and their analytical counterparts, 
$D(\xin,\yin;\xf)$ and $V(\xin,\yin;\xf)$ (see equation \eqref{eq:VD}), are plotted as functions 
of $\xf$ for (a) $\sigma=0.25$, (b) $\sigma=0.5$ and (c) $\sigma=1$.}
\label{fig:numTh21} 
\end{figure}

\begin{figure}[h!]
\begin{center}
\begin{tikzpicture}[scale=1]

\begin{groupplot}[
    group style={
        group size=2 by 2,
        horizontal sep=1.75cm,
        vertical sep=1.2cm,
    },
    width=2.2in,
    height=1.8in,
    scale only axis,
]

\nextgroupplot[
    ylabel={
    {$M_D$}
    },
    legend style={
        at={(0.02,0.98)},
        anchor=north west,
        font=\small
    },
    legend cell align=left,
    xmin=0,
    xmax=1,
    xmajorgrids,
    grid style={dashed,gray,line width=0.25pt},
    ymin=0,
    ymax=1,
    ymajorgrids,
    xlabel={{$\sigma$}}
]
\addplot[
    only marks,
    mark=o,
    mark size=2.25pt,
    line width=0.75pt,
    color=blue
]  plot coordinates{
(0.05,0.0023137374399999997) (0.1,0.010282718666666666) (0.15,0.020933414666666667) (0.2,0.04103408533333333) (0.25,0.061565082666666666) (0.3,0.090252944) (0.35,0.12256681866666666) (0.4,0.15899754133333333) (0.45,0.201555552) (0.5,0.24528288) (0.55,0.29534426666666663) (0.6,0.35378087999999996) (0.65,0.41163290666666663) (0.7,0.4747929066666666) (0.75,0.5451552799999999) (0.8,0.6086004533333333) (0.85,0.6935270933333333) (0.9,0.7618593333333333) (0.95,0.8348974666666666) (1.,0.92792216)
};
\addlegendentry{\footnotesize Numerical}

\addplot[
    orange,
    line width=1.5pt,
    samples=100,     
    domain=0:1
]{x*x};
\addlegendentry{\footnotesize Analytical}

\draw(rel axis cs: 0.925,0.075) node{\small(a)};

\nextgroupplot[
    ylabel={
    {$M_V$}
    },
    legend style={
        at={(0.02,0.98)},
        anchor=north west,
        font=\small
    },
    legend cell align=left,
    xmin=0,
    xmax=1,
    xmajorgrids,
    grid style={dashed,gray,line width=0.25pt},
    ymin=0,
    ymax=1,
    ymajorgrids,
    xlabel={{$\sigma$}},
]
\addplot[
    only marks,
    mark=o,
    mark size=2.25pt,
    line width=0.75pt,
    color=red
] plot coordinates{
(0.05,0.002424128239167565) (0.1,0.00967109290978401) (0.15,0.021845279549872074) (0.2,0.038729478206749816) (0.25,0.06010860723030072) (0.3,0.08630292992264565) (0.35,0.11648668439135423) (0.4,0.15202072342946368) (0.45,0.19167720780807132) (0.5,0.2347617810606139) (0.55,0.28114352533506076) (0.6,0.3324758890498371) (0.65,0.3868886929392844) (0.7,0.4460484900379216) (0.75,0.5048360972208638) (0.8,0.5676426557820877) (0.85,0.6360720099176128) (0.9,0.7040676885104975) (0.95,0.7830352496973106) (1.,0.858013789696254)
};
\addlegendentry{\footnotesize Numerical}

\addplot[
    green!80!black,
    line width=1.5pt,
    samples=100,     
    domain=0:1
]{x*x};
\addlegendentry{\footnotesize Analytical}

\draw(rel axis cs: 0.925,0.075) node{\small(b)};

\nextgroupplot[
    ylabel={
    {$\log_{10}(M_D)$}
    },
    legend style={
        at={(0.02,0.98)},
        anchor=north west,
        font=\small
    },
    legend cell align=left ,
    xmin=-1.5,
    xmax=0.1,
    xmajorgrids,
    grid style={dashed,gray,line width=0.25pt},
    ymin=-3.1,
    ymax=0.25,
    ymajorgrids,
    xlabel={{$\log_{10}(\sigma)$}},
]
\addplot[
    only marks,
    mark=o,
    mark size=2.25pt,
    line width=0.75pt,
    color=blue
] plot coordinates{
(-1.3010299956639813,-2.635685925777018) (-1.,-1.9878920462518572) (-0.8239087409443188,-1.6791599235995407) (-0.6989700043360187,-1.3868552427357597) (-0.6020599913279624,-1.2106655330418843) (-0.5228787452803376,-1.0445386227508198) (-0.4559319556497244,-0.9116270862645587) (-0.3979400086720376,-0.7986095913627058) (-0.3467874862246563,-0.6956052343764522) (-0.3010299956639812,-0.6103327631715132) (-0.2596373105057561,-0.5296714554112366) (-0.2218487496163564,-0.4512656421297325) (-0.18708664335714442,-0.3854899142485148) (-0.1549019599857432,-0.32349577796204765) (-0.12493873660829995,-0.2634797772794113) (-0.09691001300805639,-0.2156677284811514) (-0.07058107428570728,-0.15893656810861237) (-0.045757490560675115,-0.11812520765849764) (-0.022276394711152253,-0.0783668567260966) (0.,-0.032488453630643205)
};
\addlegendentry{\footnotesize Numerical}

\addplot[
    orange,
    line width=1.5pt,
    samples=100,     
    domain=-1.5:0.1
]{2*x};
\addlegendentry{\footnotesize Analytical}

\draw(rel axis cs: 0.925,0.075) node{\small(c)};

\nextgroupplot[
    ylabel={
    {$\log_{10}(M_V)$}
    },
    legend style={
        at={(0.98,0.02)},
        at={(0.02,0.98)},
        anchor=north west,
        font=\small
    },
    legend cell align=left,
        xmin=-1.5,
    xmax=0.1,
    xmajorgrids,
    grid style={dashed,gray,line width=0.25pt},
    ymin=-3.1,
    ymax=0.25,
    ymajorgrids,
    xlabel={{$\log_{10}(\sigma)$}},
]
\addplot[
    only marks,
    mark=o,
    mark size=2.25pt,
    line width=0.75pt,
    color=red
] plot coordinates{
(-1.3010299956639813,-2.615444409221228) (-1.,-2.0145244444414963) (-0.8239087409443188,-1.6606423933352759) (-0.6989700043360187,-1.411958354123026) (-0.6020599913279624,-1.2210633349028293) (-0.5228787452803376,-1.0639744600479677) (-0.4559319556497244,-0.9337237160617152) (-0.3979400086720376,-0.818097205097395) (-0.3467874862246563,-0.717429525680491) (-0.3010299956639812,-0.6293726042962117) (-0.2596373105057561,-0.5510719137325041) (-0.2218487496163564,-0.4782398439918073) (-0.18708664335714442,-0.4124139626174271) (-0.1549019599857432,-0.35061792646455525) (-0.12493873660829995,-0.2968495993595025) (-0.09691001300805639,-0.24592497671927985) (-0.07058107428570728,-0.19649371494977375) (-0.045757490560675115,-0.15238558612285646) (-0.022276394711152253,-0.10621868697851818) (0.,-0.06650573226683638)
};
\addlegendentry{\footnotesize Numerical}

\addplot[
    green!80!black,
    line width=1.5pt,
    samples=100,     
    domain=-1.5:0.1
]{2*x};
\addlegendentry{\footnotesize Analytical}

\draw(rel axis cs: 0.925,0.075) node{\small(d)};

\end{groupplot}

\end{tikzpicture}
\end{center}
\caption[]{Comparison between Proposition~\ref{prop:DV} and numerical simulations.  
(a,b) Numerical estimates of $M_D$ and $M_V$ (see equation~\eqref{eq:MDMV}) compared to 
their analytical counterpart $\sigma^2$ as functions of $\sigma$ for $\xf = 30$.  
(c,d) $\log_{10}(M_D)$ and $\log_{10}(M_V)$ plotted against $\log_{10}(\sigma)$, 
which theoretically yields a straight line with slope $2$.

}
\label{fig:numProp22} 
\end{figure}
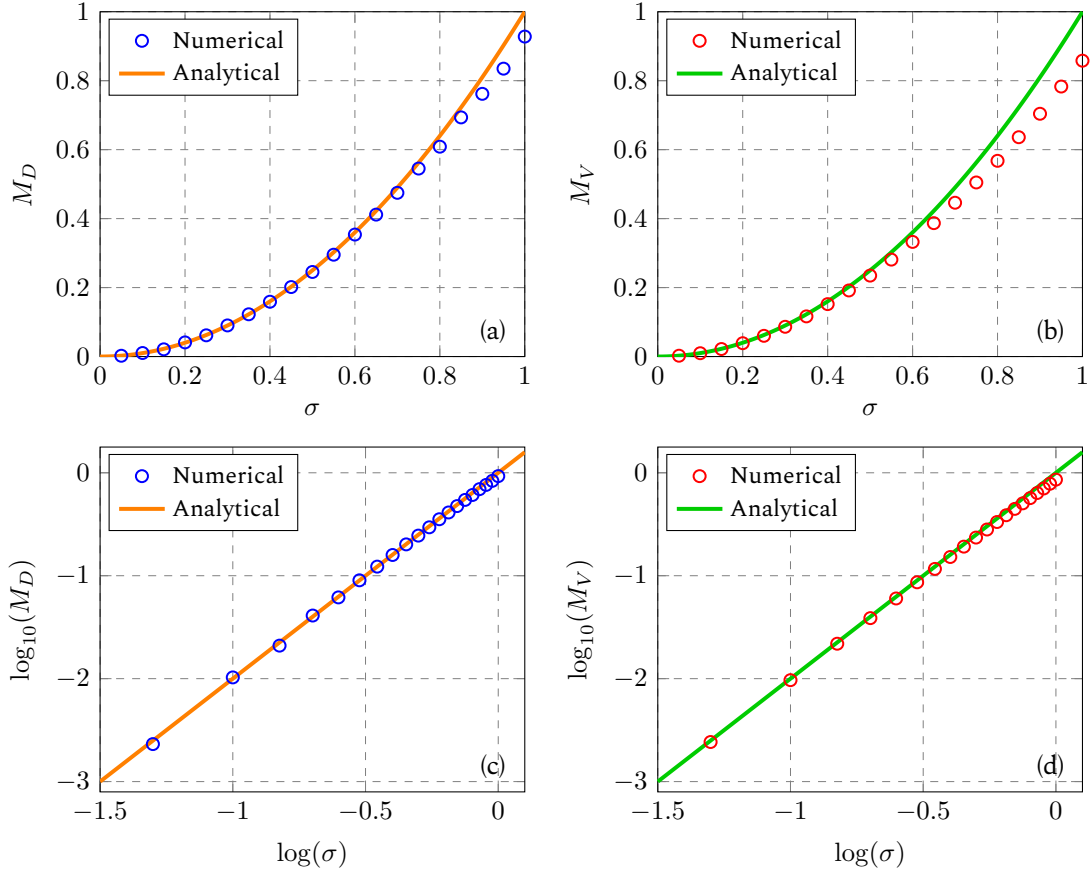

In this section the results stated in Theorem~\ref{thm:main} and 
Proposition~\ref{prop:DV} are compared with numerical simulations. 
For this purpose, the statistics $\bigexpecin{(\xin,\yin)}{y_\tau}$ 
and $\Varin{(\xin,\yin)}{y_\tau}$ are numerically estimated 
via a Monte Carlo approach for 20 noise levels $\sigma$, ranging from $0.05$ to $1$, 
with $3\cdot 10^{5}$ realizations of system~\eqref{eq:sn_scaled} generated for each value. 
Note that the indicator function $\indicator{\Omega_0}$ is not used here,
since the expectation and variance are computed numerically over all generated samples. 
This should not make any difference for small $\sigma$, since $\fP(\Omega_0^c)$ is 
exponentially small, but may have an effect for larger noise levels. 
The following initial conditions are used: $\xin=-5$ and $\yin=-\xin^2+0.1$ (i.e. near  
the critical manifold $y=-x^2$). The simulation of  the process~\eqref{eq:sn_scaled} is 
performed using an Euler-Maruyama scheme (see e.g. \cite{kloedenPlaten1992}) with an
integration time step equal to $10^{-5}$. The required standard normal random variables 
are generated using the Box–Muller algorithm.

In Figure~\ref{fig:numTh21}, the integral forms \eqref{eq:VD} of the functions 
$D(\xin,\yin;\xf)$ and $V(\xin,\yin;\xf)$ are plotted with respect to $\xf$ and 
compared with their numerically estimated counterparts 
\begin{equation}
L_D= \frac{2}{\sigma^2}\left[\bigexpecin{(\xin,\yin)}{y_\tau}-\yf\right]
 \quad\text{and}\quad
L_V= \frac{1}{\sigma^2} \Varin{(\xin,\yin)}{y_\tau}\;,
\label{eq:LDLV}
\end{equation} 
respectively.
Only the results obtained for $\sigma=0.25$ (Figure~\ref{fig:numTh21}(a)), 
$\sigma=0.5$ (Figure~\ref{fig:numTh21}(b)) and $\sigma=1$ (Figure~\ref{fig:numTh21}(c)) 
are shown here. 

Figure~\ref{fig:numProp22} displays, as functions of $\sigma$ for a fixed value $\xf=30$, 
the quantities
\begin{equation}
M_D=\frac{3}{8} \left[\bigexpecin{(\xin,\yin)}{y_\tau}-\yf\right]
\quad\text{and}\quad
M_V=\frac{2}{y^{\star}}\Varin{(\xin,\yin)}{y_\tau}\;,
\label{eq:MDMV}
\end{equation}
which, according to \eqref{eq:DV_limit}, are theoretically equal to $\sigma^2 
+ \Order{\sigma^3}$. 
The goal is to investigate the range of validity of the explicit asymptotic expressions 
for $D(\xin,\yin;\xf)$ and $V(\xin,\yin;\xf)$ given in equation~\eqref{eq:DV_limit}. 
The statistics \eqref{eq:MDMV} are first directly plotted as functions of $\sigma$ in 
Figures~\ref{fig:numProp22}(a,b) and compared to the theoretical parabola $\sigma^2$. 
Moreover, the root mean square error (RMSE) between theory and data is computed for 
both $M_D$ and $M_V$, we obtain $0.0270897$ and $0.0576924$, respectively. Finally, 
$\log_{10}(M_D)$ and $\log_{10}(M_V)$ are plotted in Figures~\ref{fig:numProp22}(c,d) 
with respect to $\log_{10}(\sigma)$ which theoretically correspond to a line with slope 
equal to 2. Again, RSME are computed between theory and data leading to $0.0185379$ and 
$0.0380149$ for $\log_{10}(M_D)$ and $\log_{10}(M_V)$, respectively.

These comparisons show excellent agreement between the analytical and numerical estimates, 
even for values of $\sigma$ that appear to lie beyond the validity range of the 
small-$\sigma$ expansions used in this study.
It should, however, be noted (see Figure~\ref{fig:numTh21}) that the discrepancy between 
the numerical and analytical results, as $\sigma$ increases, is greater when $\xf < 0$.
As it only occurs for rather large $\sigma$, however, this does not contradict our 
asymptotic results. Also, it does not correspond to the most relevant scenarios, particularly 
in the context of our SSO--NES application. The definition~\eqref{eq:LDLV} of $L_D$ 
shows that the difference between the theoretical result and the simulation depends on 
terms of order $\sigma^3$ and higher in the expansion of the expectation. Thus, in a 
sense, the remarkable fact is not that there is a discrepancy, but rather that it 
occurs only for negative values of $\xf$. A possible explanation of this observation 
is that higher-order terms in the $\sigma$-expansion of the expectation involve 
derivatives of order $3$ and higher of the deterministic transition time 
$T(\xin,\yin;\xf)$. In particular, it seems likely that the term of order 
$\sigma^3$ of the expectation will involve the integral
\begin{equation}
\label{eq:int_Tyyy} 
 \int_{\yin}^{\yf} \partial_{yyy} T(x^{\det}(y),y;\xf) \6y\;.
\end{equation} 
Figure~\ref{fig:Tderivatives} shows this quantity as a function of $\xf$, as well as  
the integral of the second derivative defining $D(\xin,\yin;\xf)$, see equation~\eqref{eq:DV}. 
Quantities depicted with solid lines are entirely computed numerically: first the 
deterministic time $T(\xin,\yin;\xf)$, then its successive derivatives with respect 
to its second argument, using finite differences, and finally the integrals from $\yin$ 
to $\yf$.While~\eqref{eq:int_Tyyy} dominates $D(\xin,\yin;\xf)$ for negative $\xf$ close 
to $\xin$, it becomes much smaller than $D(\xin,\yin;\xf)$ for positive $\xf$. This 
could explain why the second-order approximation of the expectation in $\sigma$ 
is more accurate for large positve $\xf$ than for small negative $\xf$. The figure 
also provides further validation of the analytical expression of the second derivative 
(see equation~\eqref{eq:TyTyy}), as its integral closely coincides with the equivalent 
quantity computed entirely numerically.

\begin{figure}[t]
\begin{center}
\input{TIKZ/T_deriv_integ.tex}
\end{center}
\caption[]{Integral of the second (orange) and third (green) derivatives of the 
deterministic time $T(\xin,\yin;\xf)$ with respect to its second argument. Results 
are obtained from numerical simulations of~\eqref{eq:sn_scaled} with $\sigma = 0$, 
and compared with the corresponding analytical result $(D(\xin,\yin;\xf)$~\eqref{eq:DV} 
(derived from~\eqref{eq:TyTyy}; dashed black curve shown only for the second derivative), 
plotted as a function of $\xf$ for $\xin = -2$.}
\label{fig:Tderivatives} 
\end{figure}


\section{Structure of the proofs}
\label{sec:structure} 


To prove Theorem~\ref{thm:main}, we start by fixing an initial condition 
$(\xin,\yin)$ satisfying $\xin^2 + \yin > 0$, and a final value 
$\xf > \xin$. 
The main idea is to introduce a partition 
\begin{equation}
 \xin = x_0 < x_1 < x_2 < \dots < x_{N-1} < x_N = \xf 
\end{equation} 
of the interval $[\xin,\xf]$, and to study the equation~\eqref{eq:sn_scaled} 
separately on each interval. The partition could be chosen equally spaced, 
but the proof simplifies slightly if it is chosen \lq\lq equally spaced in $y$\rq\rq, 
meaning that the deterministic solution starting in $(\xin,\yin)$ reaches $x = x_n$ 
when $y(t) = y^{\det}_n = \yin + n\gamma$, where $\gamma = (\yf - \yin)/N$
(see Figure~\ref{fig:partition} below). This means that 
the deterministic reference solution always takes the same time $\gamma$ to go from one 
section to the next one. 
Let 
\begin{equation}
 \tau_n = \inf\setsuch{t>0}{x_t = x_n}
\end{equation} 
be the first-hitting time of the line $\set{x = x_n}$. By the strong Markov 
property, to obtain the distribution of $y_\tau = y_{\tau_N}$, it is sufficient 
to know, for each $n\in\set{0,1,\dots,N-1}$, the distribution of 
$y_{\tau_{n+1}}$, conditionned on the value of $y_{\tau_n}$. 

\begin{itemize}
\item   The analysis of a single \lq\lq slice\rq\rq\ $[x_n, x_{n+1}]$ is performed in 
detail in Section~\ref{sec:one_slice}. The main result is Proposition~\ref{prop:moments_ytau}, 
which provides moment estimates related to $y_{\tau_{n+1}}$, given the value of 
$y_{\tau_n}$. To obtain these estimates, we first analyse the distribution of the 
time required to reach the section $\set{x = x_{n+1}}$, and then use this information 
to describe the exit location. The main idea in order to describe the transition 
time is to consider first a linearised process, which reduces, after a time change, 
to integrated Brownian motion (Section~\ref{ssec:scaling}). For this Gaussian process, 
we can draw on results in~\cite{Durbin85} to describe the density of transition times, using 
a fixed-point argument (Section~\ref{ssec:linearised}). Moments of the transition time can then 
be computed via the Laplace method (Section~\ref{ssec:linearised_moments}). It remains 
to transfer these results to the nonlinear equation by a comparison argument (Section~\ref{ssec:nonlinear}), and to deduce the moments of $y_\tau$, using in 
particular martingale arguments (Section~\ref{ssec:slice_moments}).

\item   Section~\ref{sec:slices} provides the proof of Theorem~\ref{thm:main} by 
combining the information obtained on each separate slice. We start by 
precisely defining the partition, as well as a likely event on which the 
analysis takes place (Section~\ref{ssec:partition}). We then obtain a 
recursive relation between moments, that is solved by going backwards in time, 
starting with the interval $[x_{N-1},x_N]$ and going all the way to $\set{x = x_0}$.
The moments can be written as sums over $n$, that are then approximated 
by integrals. Section~\ref{ssec:expectation} contains the analysis of the expectation, 
while Section~\ref{ssec:variance} is devoted to the variance.

\item   Finally, Section~\ref{sec:DV} contains an analysis of the functions $D$ and 
$V$ that appear in Theorem~\ref{thm:main}, and provides the proof of 
Proposition~\ref{prop:DV}. This relies on properties of Airy's equation and 
Airy functions, that are used to derive explicit expressions for $\partial_y T$
and $\partial_{yy}T$ using standard perturbation theory (Section~\ref{ssec:T1T2}).
Section~\ref{ssec:asympt} contains the analysis of their asymptotic expressions in 
the limit $\xf\to\infty$.
\end{itemize}


\section{Analysis of one slice}
\label{sec:one_slice} 

In this section, we analyse the SDE~\eqref{eq:sn_scaled} on one \lq\lq slice\rq\rq\ 
of width $\delta > 0$. That is, we fix an initial condition $(x_0,y_0)$ 
satisfying 
\begin{equation}
x_0^2 + y_0 > 0\;,
\end{equation} 
and define the stopping time 
\begin{equation}
\label{eq:tau_slice} 
\tau = \inf\bigsetsuch{t>0}{x_t = x_0 + \delta}\;.
\end{equation} 
We will always assume that $x_0^2 + y_0$ is bounded below by a positive constant 
independent of $\delta$. 
We consider $\delta$ as a small parameter, but such that $\sigma^2 \ll \delta$.  

The aim is to characterise the distribution of $y_\tau$, in particular its expectation 
and its variance. We divide this task into several steps:

\begin{itemize}
\item   In Section~\ref{ssec:scaling}, we perform a time change and a linearisation around 
the deterministic solution with the same initial condition, 
in order to reduce the problem to a standardised first-hitting problem, namely of the first 
time integrated Brownian motion hits a curve that crosses the $x$-axis with slope $-1$. 

\item   In Section~\ref{ssec:linearised}, we consider the linearised equation. We determine 
a sharp approximation of the density of the first-hitting time for this linearised process, 
based on results by Durbin on first-hitting densities of Gaussian processes to general, 
curved boundaries~\cite{Durbin85}. 

\item   In Section~\ref{ssec:linearised_moments}, we use the results of the 
previous section to determine the expectation and variance of 
the first-hitting time of the linearised process.

\item   In Section~\ref{ssec:nonlinear}, we quantify the error made by linearising the 
equation. 

\item   Finally, in Section~\ref{ssec:slice_moments}, we derive  properties of the 
distribution of $y_\tau$.
\end{itemize}


\subsection{Time change}
\label{ssec:scaling} 

Let 
\begin{equation}
u_t = x_t - x^{\det}(t)
\end{equation} 
denote the difference between the $x$-components of the solutions of the 
SDE~\eqref{eq:sn_scaled} and its deterministic counterpart, for the same 
initial condition $(x_0, y_0)$. It satisfies the equation 
\begin{equation}
\6u_t = (\sigma W_t + 2x^{\det}(t) u_t + u_t^2) \6t\;,
\end{equation} 
with initial condition $u_0 = 0$. The stopping time~\eqref{eq:tau_slice} 
can be rewritten 
\begin{equation}
\tau = \inf\bigsetsuch{t>0}{u_t = d(t)}\;,
\end{equation} 
where
\begin{equation}
d(t) = \delta + x_0 - x^{\det}(t)\;.
\end{equation} 
In the deterministic case $\sigma = 0$, $\tau$ is equal to the deterministic value 
\begin{align}
T(x_0, y_0)
= T(x_0, y_0; x_0 + \delta) 
&= \inf\bigsetsuch{t>0}{x^{\det}(t) = x_0 + \delta} \\
&= \inf\bigsetsuch{t>0}{d(t) = 0}\;.
\end{align}
Since
\begin{equation}
x^{\det}(t) = x_0 + (x_0^2 + y_0) t + \Order{t^2}\;,
\end{equation} 
one has 
\begin{equation}
\label{eq:d_of_t} 
d(t) = \delta - (x_0^2 + y_0) t + \Order{t^2}\;,
\end{equation} 
and therefore 
\begin{equation}
\label{eq:T_delta} 
T(x_0,y_0) = \frac{\delta}{x_0^2 + y_0} + \Order{\delta^2}\;.
\end{equation} 
We focus first on the linearised equation 
\begin{equation}
\6u^0_t = (\sigma W_t + 2x^{\det}(t) u^0_t)\6t\;,
\end{equation} 
again with initial condition $u^0_0 = 0$. 
Its solution can be written 
\begin{equation}
u^0_t = \sigma \e^{2\alpha(t)} \int_0^t \e^{-2\alpha(s)}W_s\6s\;,
\end{equation} 
where
\begin{equation}
\label{eq:alpha} 
\alpha(t) = \int_0^t x^{\det}(s) \6s 
= x_0t + \Order{t^2}\;.
\end{equation} 
Let 
\begin{align}
\tau^0 &= \inf\bigsetsuch{t>0}{u^0_t = d(t)} \\
&= \inf\Biggsetsuch{t>0}{\sigma \int_0^t \e^{-2\alpha(s)}W_s\6s 
= \e^{-2\alpha(t)}d(t)}\;. 
\end{align} 
The following result shows that up to a time change, we can always 
assume that $d(t)$ vanishes in $t = \delta$, where its slope is close to $-1$, 
and that $u^0_t$ is proportional to integrated Brownian motion. 

\begin{proposition}
\label{prop:time_change}
For sufficiently small $\delta$, there exists an invertible time change 
\begin{equation}
t = c g(\tilde t)
= c\bigbrak{\tilde t + \Order{\tilde t^2}}\;,
\end{equation} 
with $c>0$ and $g$ a strictly increasing function, close to identity, such that 
\begin{equation}
\label{eq:tau0_tildetau0} 
\tau^0 = cg(\tilde\tau^0)\;,
\end{equation} 
where
\begin{equation}
\tilde\tau^0 
= \inf\bigsetsuch{\tilde t>0}{\tilde\sigma X_{\tilde t} = \tilde d(\tilde t)}\;.
\label{eq:tilde_tau0} 
\end{equation} 
Here 
\begin{equation}
\label{eq:def_sigma_tilde} 
\tilde\sigma = \sigma c^{3/2}(1 + \Order{\delta})\;,
\end{equation} 
$X_{\tilde t}$ is the integrated Brownian motion 
\begin{equation}
X_{\tilde t} = \int_0^{\tilde t} W_s \6s\;,
\end{equation} 
and $\tilde d$ is a strictly decreasing function satisfying 
\begin{equation}
\tilde d(\delta) = 0\;, \qquad 
\tilde d'(\delta) = -1\;.
\label{eq:dtilde} 
\end{equation} 
Furthermore, $\tilde d'(\tilde t) = -1 + \Order{\delta}$ whenever $\tilde t = \Order{\delta}$.
\end{proposition}
\begin{proof}
Given a constant $c > 0$, that will be chosen below, let $g:\R_+\to\R_+$ 
be the solution of the ODE 
\begin{equation}
\label{eq:ODE_g} 
g'(\tilde s) 
= \sqrt{\frac{\tilde s}{g(\tilde s)}} \e^{2\alpha(cg(\tilde s)))}
\end{equation} 
with initial condition $g(0) = 0$. Separating variables in~\eqref{eq:ODE_g}, we get
\begin{equation}
\int_0^{g(\tilde s)} \sqrt{g} \e^{-2\alpha(cg)} \6g
= \int_0^{\tilde s} \sqrt{\tilde s} \6\tilde s
= \frac23 \tilde s^{3/2}\;.
\end{equation} 
Since the left-hand side is strictly increasing in $g(\tilde s)$, this equation 
has a unique solution for any $\tilde s \geqs 0$, which is strictly increasing 
in $\tilde s$. We will later use the fact that~\eqref{eq:alpha} and a Taylor 
expansion imply 
\begin{equation}
\label{eq:g}
g(\tilde s) = \tilde s + \frac45 x_0 c\tilde s^2 + \Order{\tilde s^3}\;.
\end{equation} 
Since $g$ in strictly increasing, we can set $\tau^0 = cg(\tilde\tau^0)$, where 
\begin{equation}
\tilde\tau^0 
= \inf\bigsetsuch{\tilde t>0}{\sigma \tilde X_{\tilde t} = \tilde d(\tilde t)}\;,
\label{eq:tilde_tau0_2} 
\end{equation} 
with 
\begin{align}
\widetilde X_{\tilde t} &= \int_0^{cg(\tilde t)} \e^{-2\alpha(s)} W_s \6s\;, \\
\tilde d(\tilde t) &= \e^{-2\alpha(cg(\tilde t))} d(cg(\tilde t))\;.
\end{align}
We now perform the change of variables $s = cg(\tilde s)$ in the integral 
defining $\widetilde X_{\tilde t}$. We have $\6s = cg'(\tilde s)\6\tilde s$, 
and the scaling property of Brownian motion implies that 
\begin{equation}
W_{cg(\tilde s)} = \sqrt{\frac{cg(\tilde s)}{\tilde s}} W_{\tilde s}\;,
\end{equation} 
where equality is in distribution. Therefore, \eqref{eq:ODE_g} implies
\begin{equation}
\sigma\widetilde X_{\tilde t}
= \sigmac^{3/2} \int_0^{\tilde t} g'(\tilde s) \sqrt{\frac{g(\tilde s)}{\tilde s}}
\e^{-2\alpha(cg(\tilde s)))} W_{\tilde s} \6\tilde s
= c^{3/2} \sigma X_{\tilde t}\;,
\end{equation} 
so that~\eqref{eq:tilde_tau0_2} is equivalent to~\eqref{eq:tilde_tau0} with $\tilde\sigma 
= c^{3/2} \sigma$. 
It remains to choose the value of $c$ in such a way that $\tilde d$ 
satisfies~\eqref{eq:dtilde}. Setting 
\begin{equation}
\label{eq:def_c} 
c = \frac{T(x_0,y_0)}{g(\delta)}
= \frac{1}{x_0^2 + y_0} + \Order{\delta} 
\end{equation} 
entails $d(cg(\delta)) = 0$, and therefore $\tilde d(\delta) = 0$. Furthermore, 
we obtain 
\begin{align}
\tilde d'(\delta) 
&= c \e^{-2\alpha(cg(\delta))} d'(cg(\delta)) g'(\delta) \\
&= c d'(T(x_0,y_0)) (1 + \Order{\delta}) \\
&= -1 + \Order{\delta}\;,
\end{align}
where we have used~\eqref{eq:d_of_t} and~\eqref{eq:alpha}.
A similar computation shows that $\tilde d'(\tilde t)$ has the same 
order whenever $\tilde t = \Order{\delta}$.
We can achieve $\tilde d'(\delta) = -1$ by scaling $\tilde d(\tilde t)$ by a 
factor $1 + \Order{\delta}$, and modifying the definition of $\tilde\sigma$ 
by an equivalent amount. 
\end{proof}


\subsection{Linearised equation: density of hitting times}
\label{ssec:linearised} 

Thanks to Proposition~\ref{prop:time_change}, we have reduced the problem to 
the simpler one of characterising the first time integrated Brownian motion 
crosses a decreasing curve, intersecting the time axis with slope $-1$. 
Dropping the tildes for now, we can rewrite this as characterising 
\begin{equation}
\tau^0 = \inf\bigsetsuch{t>0}{\sigma X_t = d(t)}\;,
\end{equation} 
where $X_t$ is the integrated Brownian motion
\begin{equation}
\label{eq:def_Xt}
X_t = \int_0^t W_s \6s\;,
\end{equation} 
and $d$ is strictly decreasing, with $d(\delta) = 0$ and 
$d'(\delta) = -1$. As a consequence, we have 
\begin{equation}
\label{eq:doft} 
d(t) = -(t-\delta) + \Order{(t-\delta)^2}\;.
\end{equation} 
We will rely for that on results 
in~\cite{Durbin85} on first-passage densities of Gaussian processes through 
a time-dependent boundary. These results depend on the covariance function of the 
process, which we compute in the following lemma.

\begin{lemma}
\label{lem:covariance}
For any choice of\/ $0\leqs s \leqs u \leqs t$, one has 
\begin{align}
\rho(s,u,t) 
:={}& \bigexpec{(X_u-X_s)(X_t - X_s)} \\
={}& \frac23 s^3 - \frac16 u^3 + \frac12 \bigbrak{u^2t - s^2u - s^2t}\;.
\end{align}
\end{lemma}
\begin{proof}
This follows by writing
\begin{equation}
\rho(s,u,t) = \int_s^u \int_s^t \bigexpec{W_{v_1}W_{v_2}} \6v_1 \6v_2\;,
\end{equation} 
using the property $\bigexpec{W_{v_1}W_{v_2}} = v_1 \wedge v_2$ of Brownian motion,  
and splitting the integral along the line $\set{v_1 = v_2}$. 
\end{proof}

In particular, the covariance and variance of $X_t$ are given by 
\begin{align}
\bigexpec{X_sX_t} &= \rho(s,t) := \rho(0,s,t) = -\frac16s^3 + \frac12 s^2 t\;, \\
\bigexpec{X_t^2} &= \rho(t,t) = \frac13 t^3\;.
\end{align}
The distribution of $X_t$ is centred Gaussian with this covariance, and therefore 
the density of $u^0_t$ at $d(t)$ is given by 
\begin{equation}
\label{eq:def_phi} 
\phi(t) = \frac{1}{\sigma\sqrt{2\pi\rho(t,t)}}
\exp \biggset{-\frac{d(t)^2}{2\sigma^2 \rho(t,t)}}\;.
\end{equation} 
By~\cite[Eq.~(2)]{Durbin85}, the density of the first-passage time $\tau^0$ can 
be written as 
\begin{equation}
\label{eq:psi_b_phi} 
\psi(t) = b(t)\phi(t)\;,
\end{equation} 
where the function $b(t)$ solves a certain fixed-point equation, and is well-approximated 
by 
\begin{equation}
b_0(t) 
= \frac{d(t)}{\rho(t,t)} \partial_s \rho(t,t) - d'(t) 
= \frac{3}{2t} d(t) - d'(t)\;.
\end{equation} 
More precisely, we have from~\cite[Eq.~(7)]{Durbin85} that 
\begin{equation}
\label{eq:psi} 
\psi(t) = b_0(t)\phi(t) + \int_0^t \tilde b(s,t) \phi(t\vert s)\psi(s)\6s\;, 
\end{equation} 
where 
\begin{equation}
\label{eq:phi_ts} 
\phi(t\vert s) 
= \frac{1}{\sigma\sqrt{2\pi\rho(s,t,t)}}
\exp \biggset{-\frac{[d(t)-d(s)]^2}{2\sigma^2 \rho(s,t,t)}}
\end{equation} 
is analogous to $\phi(t)$, but for a starting point $(s,d(s))$ with $s\in[0,t]$, and 
\begin{equation}
\label{eq:btilde} 
\tilde b(s,t) = d'(t) - \beta_1(s,t)d(s) - \beta_2(s,t)d(t)\;, 
\end{equation} 
where $\beta_1$ and $\beta_2$ solve the linear system 
\begin{equation}
\label{eq:beta} 
\begin{pmatrix}
\rho(s,s) & \rho(s,t) \\
\rho(s,t) & \rho(t,t)
\end{pmatrix}
\begin{pmatrix}
\beta_1(s,t) \\ 
\beta_2(s,t) 
\end{pmatrix}
=
\begin{pmatrix}
\partial_t \rho(s,t) \\
\partial_s \rho(t,t)
\end{pmatrix}\;.
\end{equation} 
Substituting~\eqref{eq:phi_ts} in~\eqref{eq:psi} and using~\eqref{eq:psi_b_phi}, we obtain that $b(t)$ 
satisfies the fixed-point equation 
\begin{equation}
\label{eq:b} 
b(t) = b_0(t) + \frac{1}{\sigma\sqrt{2\pi}}
\int_0^t \sqrt{\frac{\rho(t,t)}{\rho(s,t,t)\rho(s,s)}} \tilde b(s,t) b(s) 
\e^{-r(s,t)/(2\sigma^2)} \6s\;,
\end{equation} 
where
\begin{equation}
r(s,t) = \frac{d(s)^2}{\rho(s,s)} - \frac{d(t)^2}{\rho(t,t)} 
+ \frac{[d(t) - d(s)]^2}{\rho(s,t,t)}\;.
\end{equation} 
We will use Banach's fixed point theorem to show that the solution $b$ of~\eqref{eq:b} 
is indeed close to $b_0$. For this, the following two lemmas will be useful.

\begin{lemma}
\label{lem:rmin} 
For sufficiently small $\delta$, there exists a constant $r_{\min} > 0$, 
independent of $\delta$, $s$ and $t$, such that 
\begin{equation}
r(s,t) \geqs \frac{r_{\min}t^2}{s^3}
\geqs \frac{r_{\min}}{s}
\end{equation} 
holds whenever\/ $0\leqs s \leqs t \leqs T_{\max} = 2\delta$. 
\end{lemma}
\begin{proof}
We first note that Lemma~\ref{lem:covariance} implies that 
\begin{equation}
\label{eq:rho_stt} 
\rho(s,t,t) 
= \frac23 s^3 + \frac13 t^3 - s^2 t 
= \frac13 (t-s)^2 (t+2s)\;.
\end{equation} 
It follows that 
\begin{equation}
r(s,t) 
= 3 \biggbrak{\frac{d(s)^2}{s^3} - \frac{d(t)^2}{t^3} 
+ \frac{[d(t)-d(s)]^2}{(t-s)^2(t+2s)}}\;.
\end{equation} 
It follows from the mean value theorem and the fact that $d'(t) = -1 + \Order{\delta}$ 
that 
\begin{equation}
\frac{[d(t)-d(s)]^2}{(t-s)^2(t+2s)}
= \frac{1+\Order{\delta}}{t+2s}\;.
\end{equation} 
This shows in particular that $r(t,t)$ is well-defined (it is equal to $d'(t)^2/t$). 
It will be convenient to map the triangular domain $\set{0\leqs s \leqs t \leqs 2\delta}$ 
to the rectangle $\set{0\leqs x\leqs 2, 0\leqs p \leqs 1}$, using the change of variables
\begin{equation}
t = \delta x\;, \qquad 
s = \delta p x\;. 
\label{eq:ts_xk} 
\end{equation} 
Using the fact that $t+2s = \delta x(1+2p) \leqs 3\delta x$
and $d(t)^2 = (\delta-t)^2(1+\Order{\delta})$, we obtain
\begin{equation}
r(\delta x, \delta p x) \geqs \frac{3}{\delta p^3 x^3} g(x, p)
(1+\Order{\delta})\;,
\label{eq:r_scaled} 
\end{equation} 
where
\begin{align}
g(x,p) &= (1-px)^2 - p^3(1-x)^2 + \frac{p^3x^2}{3} \\
&= ax^2 - 2bx + c\;,
\end{align}
with 
$a = p^2(1 - \tfrac23 p)$, 
$b = p(1-p^2)$, 
and $c = 1-p^3$.
One can easily compute the minimum of $g(x,p)$ over $x\in[0,2]$. However, it turns 
out to be preferable to minimise $g(x,p)/x^2$. One finds that 
\begin{equation}
\frac{\partial}{\partial x} \biggpar{\frac{g(x,p)}{x^2}}
= \frac{2(bx-c)}{x^3}
\end{equation} 
vanishes in 
\begin{equation}
x^* = \frac{c}{b} = \frac{p^2+p+1}{p(1+p)} = 1 + \frac{1}{p^2 + p}\;.
\end{equation} 
There are two cases to consider:
\begin{enumerate}
\item  If $p \leqs p^* = \frac{\sqrt{5}-1}{2}$, then $x^* \geqs 2$, so that 
$g(x,p)/x^2$ is decreasing for $x\in[0,2]$, and therefore 
\begin{equation}
\frac{g(x,p)}{x^2} 
\geqs \frac{g(2,p)}{4} 
= \frac14 \Bigbrak{\frac13 p^3 + (1-2p)^2}\;.
\end{equation} 
One easily checks that this is bounded below by a positive constant 
for $p\in[0,p^*]$. 

\item   If $p^* < p \leqs 1$,  then $x^* \in (0,2)$ and 
\begin{equation}
\frac{g(x,p)}{x^2} 
\geqs \frac{g(x^*,p)}{(x^*)^2} 
= a - \frac{b^2}{c} 
= p^2 \biggbrak{1 - \frac23 p - \frac{(1-p^2)(1+p)}{1+p+p^2}}\;,
\end{equation} 
which is also bounded below by a positive constant 
for $p\in[p^*,1]$. 
\end{enumerate}
The result follows by replacing this lower bound in~\eqref{eq:r_scaled}, together with 
the relations 
\begin{equation}
x = \frac{t}{\delta}\;, \qquad 
p = \frac{s}{\delta x} = \frac{s}{t}
\end{equation} 
which are the inverse of the transformation~\eqref{eq:ts_xk}. 
\end{proof}

The next lemma gives a bound on the prefactor of the exponential term in~\eqref{eq:b}.

\begin{lemma}
\label{lem:prefactor}
There exists a constant $M_0 > 0$ such that for any $0 < s \leqs t \leqs T_{\max}$, 
one has 
\begin{equation}
\Biggabs{\sqrt{\frac{\rho(t,t)}{\rho(s,t,t)\rho(s,s)}} \tilde b(s,t)} 
\leqs \frac{M_0\delta}{s^{5/2}}\;.
\end{equation} 
\end{lemma}
\begin{proof}
The matrix $M(s,t)$ appearing in~\eqref{eq:beta} has determinant 
\begin{equation}
\det M(s,t) = \frac{s^3}{36}
\bigbrak{4t^3 - 9st^2 + 6s^2t - s^3} 
= \frac{1}{36}s^3(t-s)^2 (4t-s)\;.
\end{equation} 
This allows to solve~\eqref{eq:beta} for $\beta_1$ and $\beta_2$. 
Using the fact that 
\begin{equation}
\partial_t \rho(s,t) = \frac12 s^2\;, \qquad 
\partial_s \rho(t,t) = \frac12 s^2\;, 
\end{equation} 
we obtain 
\begin{equation}
\beta_1(s,t) = -\frac{3t^2}{s(t-s)(4t-s)}\;, \qquad 
\beta_2(s,t) = \frac{3(2t-s)}{(t-s)(4t-s)}\;.
\end{equation} 
By Taylor's formula, there exists $\theta\in(s,t)$ such that 
\begin{equation}
d(s) = d(t) - (t-s) d'(t) + \frac12(t-s)^2 d''(\theta)\;.
\end{equation}
Plugging the last two relations into~\eqref{eq:btilde} and rearranging 
terms yields 
\begin{equation}
\tilde b(s,t) 
= \frac{t-s}{s(4t-s)}
\biggbrak{3d(t) - (3t-s) d'(t) + \frac32 t^2 d''(\theta)}\;.
\end{equation} 
Using $4t-s\geqs 3t$ and $3t-s \leqs 3t$, we obtain 
\begin{equation}
\abs{\tilde b(s,t)} 
\leqs \frac{t-s}{st} \Bigbrak{\abs{d(t)} + t\abs{d'(t)} + \frac12t^2 \abs{d''(\theta)}}\;.  
\end{equation} 
The result then follows by using the lower bound 
$\rho(s,t,t) \geqs \frac13 (t-s)^2t$, which follows from~\eqref{eq:rho_stt}, 
together with the facts that $d(t) = \Order{\delta}$, 
while $d'$ and $d''$ are bounded on $[0,T_{\max}]$. 
\end{proof}

We can now show the main result of this subsection, using a classical fixed-point argument. 

\begin{proposition}
For sufficiently small $\sigma$, there exists a constant $\kappa > 0$ such 
that~\eqref{eq:b} has a unique solution satisfying 
\begin{equation}
\bigabs{b(t) - b_0(t)} \leqs \frac{\delta^3}{t} \e^{-\kappa/\sigma^2} 
\end{equation} 
for all $t\in[0,T_{\max}]$. 
\end{proposition}
\begin{proof}
We will apply Banach's fixed-point theorem to the space 
$\cB = \cC([0,T_{\max}],\R)$ of continuous functions $b:[0,T_{\max}]\to\R$, 
equipped with the weighted supremum norm 
\begin{equation}
\norm{b} = \sup_{0\leqs t\leqs T_{\max}} t\abs{b(t)}\;.
\end{equation} 
This definition guarantees that $\norm{b_0} \leqs M_1\delta$ for a constant 
$M_1$ independent of $t$ and $\delta$. 
We write \eqref{eq:b} as the fixed-point equation $b = \Gamma b$, where 
$\Gamma:\cB\to\cB$ is the integral operator 
\begin{equation}
(\Gamma b)(t) = b_0(t) + \int_0^t k(s,t)b(s)\6s\;,
\end{equation} 
with 
\begin{equation}
k(s,t) = \frac{1}{\sigma\sqrt{2\pi}}
\sqrt{\frac{\rho(t,t)}{\rho(s,t,t)\rho(s,s)}} \tilde b(s,t)
\e^{-r(s,t)/(2\sigma^2)}\;.
\end{equation} 
Lemmas~\ref{lem:rmin} and~\ref{lem:prefactor} imply that 
\begin{equation}
\bigabs{k(s,t)} \leqs \frac{M_0\delta}{\sigma\sqrt{2\pi}s^{5/2}}
\e^{-r_{\min}/(2\sigma^2 s)}\;.
\end{equation} 
It follows that 
\begin{equation}
\bigabs{(\Gamma b)(t)} 
\leqs \frac{M_1\delta}{t} 
+ \frac{M_0\delta}{\sqrt{2\pi}}
\int_0^t \frac{\e^{-r_{\min}/(2\sigma^2 s)}}{\sigma s^{7/2}} \6s \norm{b}\;.
\end{equation} 
Using the change of variables $x = r_{\min}/(2\sigma^2 s)$, one checks 
that the integral is bounded by $\e^{-\kappa/\sigma^2}$ for a constant $\kappa$ 
comparable to $r_{\min}$. This implies 
\begin{equation}
\norm{\Gamma b} \leqs M_2 \delta 
\bigbrak{1 + \delta\e^{-\kappa/\sigma^2} \norm{b}}
\end{equation} 
for a constant $M_2 > 0$. Choosing for instance $R = 2M_2\delta$, we see 
that for $\norm{b}\leqs R$ and $\sigma$ small enough, one has 
$\norm{\Gamma b} \leqs R$, so that $\Gamma$ maps the ball of radius $R$ 
into itself. 

A similar argument shows that if $b_1$ and $b_2$ belong to that ball, then 
\begin{equation}
\norm{\Gamma b_2 - \Gamma b_1} \leqs M_2 \delta^2\e^{-\kappa/\sigma^2} \norm{b_2-b_1}\;.
\end{equation} 
Therefore, for sufficiently small $\sigma$, $\Gamma$ is a contraction
on the ball $\set{\norm{b}\leqs R}$, 
with contraction constant $\lambda = M_2 \delta^2\e^{-\kappa/\sigma^2} < 1$, 
so that Banach's fixed point theorem implies that the equation $\Gamma b = b$ 
admits a unique solution. Furthermore, the sequence of $b_n = \Gamma^n b_0$ 
converges to the fixed point $b$, so that 
\begin{align}
\norm{b_0 - b}
&\leqs \norm{b_0 - b_1} + \norm{b_1 - b_2} + \dots \\
&\leqs (1 + \lambda + \lambda^2 + \dots) \norm{b_0 - b_1} \\
&= \frac{1}{1-\lambda} \norm{b_0 - b_1}\;.
\end{align}
Since $b_0 = \Gamma(0)$, we have 
$\norm{b_0 - b_1} \leqs \lambda \norm{b_0} \leqs M_1\delta\lambda$. 
By slightly decreasing the value of $\kappa$, we can absorb any multiplicative 
constant in the term $\e^{-\kappa/\sigma^2}$, so that 
$\norm{b_0-b} \leqs \delta^3\e^{-\kappa/\sigma^2}$, 
which implies the result.
\end{proof}

Owing to~\eqref{eq:psi}, a direct consequence of this result is that the 
density of $\tau^0$ is given on $[0,T_{\max}]$ by 
\begin{equation}
\label{eq:psi2} 
\psi(t) 
= \biggbrak{\frac{3}{2t}d(t) - d'(t) 
+ \biggOrder{\frac{\delta^3}{t}\e^{-\kappa/\sigma^2}}} \phi(t)\;.
\end{equation} 
This will allow us to compute moments of $\tau^0$.


\subsection{Linearised equation: moments of hitting times}
\label{ssec:linearised_moments} 

We now compute the expectation and variance of the first hitting time $\tau^0$, 
using the density obtained in the previous section. To this end, we introduce the event
\begin{equation}
 \label{eq:def_Omega1}
 \Omega_1 = \bigsetsuch{\omega}{\tau^0(\omega) < T_{\max}}\;,
\end{equation} 
where we recall that $T_{\max} = 2\delta$. 
We will show in Corollary~\ref{cor:Omega1} below that $\Omega_1$ has a probability 
exponentially close to $1$. 

\begin{proposition}
\label{prop:moments_tau0}  
For sufficiently small $\sigma$ and $\delta$, the first two moments of 
$\tau^0\indicator{\Omega_1}$ admit the expansions
\begin{equation}
\label{eq:bound_Etau0}
\bigexpec{\tau^0 \indicator{\Omega_1}}
= \delta+\frac12 \delta^2\sigma^2+\Order{\delta^3\sigma^2}\;,
\end{equation}
and \begin{equation}
\label{eq:bound_Vtau0}
\variance\bigbrak{\tau^0\indicator{\Omega_1}}
= \frac13 \delta^3 \sigma^2 +\Order{\delta^4\sigma^2}\;.
\end{equation}
Furthermore, the third moment satisfies 
\begin{equation}
 \label{eq:bound_Etau03}
\bigexpec{(\tau^0)^3 \indicator{\Omega_1}}
= \delta^3 + \Order{\delta^4\sigma^2}\;.
\end{equation} 
\end{proposition}
\begin{proof}
The expected value of $\tau^0 \indicator{\Omega_1}$ is given by the integral
\begin{equation}
\label{eq:moments_tau0_integral1} 
\expec{\tau^0 \indicator{\Omega_1}} = \int_0^{T_{\max}} t \, \psi(t) \6t\;,
\end{equation}
where~\eqref{eq:psi2}, \eqref{eq:doft} and~\eqref{eq:def_phi} imply that the density $\psi$ 
is given by  
\begin{equation}
\psi(t) = \frac{\sqrt{3}}{\sigma\sqrt{2\pi t^3}}
\biggbrak{\biggpar{\frac{3\delta}{2t}-\frac12}
+\Order{\delta-t}+\biggOrder{\frac{\delta^3}{t}\e^{-\kappa/\sigma^2}}}
\exp \biggset{-\frac{3d(t)^2}{2\sigma^2 t^3}}\;.
\end{equation}
Substituting this in the integral~\eqref{eq:moments_tau0_integral1} gives 
\begin{equation}
\expec{\tau^0 \indicator{\Omega_1}}
=\frac{1}{\sigma}\sqrt{\frac{3}{2\pi}}\int_0^{T_{\max}}
\biggbrak{q_\delta(t)+r_1(t)+\bar r_1(t)}
\exp \biggset{-\frac{p_\delta(t)}{\sigma^2}(1+r_2(t))} \6t\;,
\end{equation}
where 
\begin{equation}
q_\delta(t)=\frac{1}{\sqrt{t}}\biggpar{\frac{3\delta}{2t}-\frac12}\;, \qquad  
p_\delta(t)=\frac32 \frac{(\delta-t)^2}{t^3}\;,
\end{equation}
and the remainders satisfy 
$r_1(t), r_2(t) = \Order{\delta-t}$ and 
$\bar r_1(t) = \Order{\delta^3 t^{-1} \e^{-\kappa/\sigma^2}}$. 
We note that the remainder $\bar r_1(t)$ yields only an exponentially small error term, 
despite the fact that it diverges like $t^{-1}$, thanks to the behaviour in $t^{-3}$ 
of $p_\delta(t)$. Since this is negligible with respect to any algebraic order term 
in $\sigma$, we can ignore it from now on. 
We now perform the change of variables
\begin{equation}
z=\frac{\sqrt{3}}{\delta^{3/2}\sigma }(\delta-t)\;,
\end{equation}
which yields 
\begin{equation}
\begin{split}
 \expec{\tau^0 \indicator{\Omega_1}}
=\frac{\delta^{3/2}}{\sqrt{2\pi}}
\int_{-\sqrt{3}/(\sqrt{\delta}\sigma)}^{\sqrt{3}/(\sqrt{\delta}\sigma)}
&\biggbrak{q_\delta\biggpar{\delta-\frac{\delta^{3/2}\sigma}{\sqrt{3}}z}
+\tilde r_1(z)}\\
&\times\exp\biggset{-\frac{1}{\sigma^2}p_\delta
\biggpar{\delta-\frac{\delta^{3/2}\sigma}{\sqrt{3}}z}
\bigpar{1+\tilde r_2(z)}}\6z\;,
\end{split}
\end{equation}
with remainders $\tilde r_1(z), \tilde r_2(z) = \Order{\delta^{3/2}\sigma z}$. 
To simplify this expression, we introduce 
\begin{equation}
\label{eq:def_sigmatilde} 
\tilde{\sigma}=\frac{\delta^{1/2}\sigma}{\sqrt{3}}\;,
\end{equation}
and note that
\begin{equation}
q_\delta\biggpar{\delta-\frac{\delta^{3/2}\sigma}{\sqrt{3}}z}
= q_\delta \bigpar{\delta(1-\tilde\sigma z)}
=\frac{1}{\sqrt{\delta}}q_1\bigpar{1-\tilde\sigma z}\;,
\end{equation}
and 
\begin{equation}
p_\delta\biggpar{\delta-\frac{\delta^{3/2}\sigma}{\sqrt{3}}z}
= \frac{z^2}{2}\frac{\delta^3}{(\delta-\delta\tilde\sigma z)^3}\sigma^2
= \frac{z^2}{2}\frac{1}{(1 - \tilde\sigma z)^3}\sigma^2\;.
\end{equation}
The expectation then becomes
\begin{equation}
\label{eq:expec_tau^0}
\expec{\tau^0 \indicator{\Omega_1}}
= \frac{\delta}{\sqrt{2\pi}}\int_{-1/\tilde{\sigma}}^{1/\tilde{\sigma}}
\bigbrak{q_1(1-\tilde{\sigma}z)+\tilde r_1(z)}
\exp\biggset{-\frac{z^2}{2}\frac{1+\tilde r_2(z)}{(1-\tilde{\sigma}z)^3}}\6z\;.
\end{equation}
For sufficiently small $\tilde\sigma$, we can use the Taylor expansions 
\begin{equation}
q_1(1-\tilde{\sigma}z)
=q_1(1)-q'_1(1)\tilde{\sigma}z+\frac12 q''_1(1)\tilde{\sigma}^2z^2
+\Order{\tilde{\sigma}^3z^3}\;,
\end{equation}
and 
\begin{align}
\exp\biggset{-\frac{z^2}{2}\frac{1+\tilde r_2(z)}{(1-\tilde{\sigma}z)^3}}
&=\exp\biggset{-\frac{z^2}{2}(1+3\tilde{\sigma}z
+6\tilde{\sigma}^2z^2+\Order{\tilde{\sigma}^3z^3})(1+\tilde r_2(z))}\\
&=\exp\biggset{-\frac32\tilde{\sigma}z^3
-3\tilde{\sigma}^2z^4+\Order{\tilde{\sigma}^3z^5} + \bigOrder{z^2\tilde r_2(z)}}
\e^{-z^2/2}\\
&=\biggbrak{1-\frac32z^3\tilde{\sigma}
+\biggpar{-3z^4+\frac98z^6}\tilde{\sigma}^2
+\Order{\tilde{\sigma}^3z^5} + \bigOrder{z^2\tilde r_2(z)}}
\e^{-z^2/2}\;.
\end{align}
Using the fact that $\tilde r_i(z) = a_i\delta\tilde\sigma z + 
\Order{\delta^2\tilde\sigma^2z^2}$, $i=1,2$, for some constants $a_1, a_2$, 
as well as the expression
\begin{equation}
\frac{1}{\sqrt{2\pi}}
\int_{-\infty}^{\infty}z^k\e^{-z^2/2}\6z = 
\begin{cases}
\prod_{i=1}^{k/2} (2i-1)
& \text{if $k$ is even} \;, \\
0 &  \text{if $k$ is odd}\;,
\end{cases}
\end{equation}
for the moments of the standard normal law, 
we find, upon substituting the expansions into~\eqref{eq:expec_tau^0},  
\begin{equation}
\expec{\tau^0 \indicator{\Omega_1}}
=\delta\biggbrak{q_1(1)
+\biggpar{\frac{63}{8}q_1(1)+\frac92 q'_1(1)
+\frac12 q''_1(1)}\tilde{\sigma}^2+\Order{\delta\tilde{\sigma}^2}}\;.
\end{equation}
Here we used the fact that extending the bounds in the integral~\eqref{eq:expec_tau^0} to 
$\pm\infty$ only produces exponentially small error terms. 
The Taylor coefficients of $q_1$ at $1$ are found to be 
$q_1(1)=1$, $q'_1(1)=-2$, and $q''_1(1)=\frac{21}{4}$,
which yields 
\begin{equation}
\expec{\tau^0 \indicator{\Omega_1}}=\delta+\frac32\delta\tilde{\sigma}^2+\Order{\delta^2\tilde{\sigma}^2}\,.
\end{equation}
Plugging in the value~\eqref{eq:def_sigmatilde} of $\tilde{\sigma}$ 
gives~\eqref{eq:bound_Etau0}.
A similar computation, in which only the expression for $q_\delta$ is changed, 
can be made for the second moment of $\tau^0$. One obtains 
\begin{equation}
\expec{(\tau^0)^2\indicator{\Omega_1}}
=\delta^2+\frac43 \delta^3 \sigma^2 +\Order{\delta^4\sigma^2}\;,
\end{equation}
which yields~\eqref{eq:bound_Vtau0}. 
The proof of~\eqref{eq:bound_Etau03} is similar. 
\end{proof}

Recall that the moments of $\tau^0$ have been computed after performing the time change defined 
in Section~\ref{ssec:scaling} and dropping the tildes. We now have to undo this time change. 
We first note that the event $\Omega_1$ becomes 
\begin{equation}
 \Omega_1 = \bigset{\tau^0(\omega) < \widehat T_{\max}}\;,
 \qquad 
 \text{where $\widehat T_{\max} = cg(T_{\max})$\;.}
\end{equation} 
The result on moments is then as follows, where we introduced the shorthand 
\begin{equation}
 r_0 = x_0^2 + y_0\;.
\end{equation} 

\begin{corollary}
\label{cor:moments_tau0}
For sufficiently small $\sigma$ and $\delta$, the moments of $\tau^0$ before the time change 
satisfy 
\begin{align}
 \bigexpec{\tau^0\indicator{\Omega_1}}
 &= T(x_0,y_0) + \frac12 \frac{T(x_0,y_0)^2\sigma^2}{r_0^2}
 + \biggOrder{\frac{\delta^3\sigma^2}{r_0^4}}\;, \\
 \variance\bigbrak{\tau^0\indicator{\Omega_1}}
 &= \frac13 \frac{T(x_0,y_0)^3\sigma^2}{r_0^2}
 + \biggOrder{\frac{\delta^4\sigma^2}{r_0^5}}\;, \\
 \bigexpec{(\tau^0)^3\indicator{\Omega_1}}
 &= T(x_0,y_0)^3 + \biggOrder{\frac{\delta^4\sigma^2}{r_0^6}}\;.
\end{align}
\end{corollary}
\begin{proof}
Recall the relation $\tau^0 = cg(\tilde\tau^0)$, cf.~\eqref{eq:tau0_tildetau0},  
where $g(t) = t+\Order{t^2}$, cf.~\eqref{eq:g}, and 
the constant $c$ is given by~\eqref{eq:def_c}. 
To compute the expectation of $\tau^0$, we write 
\begin{align}
\bigexpec{\tau^0\indicator{\Omega_1}} - cg(\delta) \fP(\Omega_1) 
&= \bigexpec{(\tau^0 - cg(\delta))\indicator{\Omega_1}} \\
&= c\bigexpec{(g(\tilde\tau^0) - g(\delta))\indicator{\Omega_1}} \\
&= c\biggexpec{\int_\delta^{\tilde\tau_0}g'(t)\6t\,\indicator{\Omega_1}} \\
&= c\Bigexpec{\bigpar{\tilde\tau^0 - \delta + \Order{(\tilde\tau^0)^2 - \delta^2}}
\indicator{\Omega_1}} \\
&= c \biggbrak{\delta + \frac12\delta^2\tilde\sigma^2 + \Order{\delta^3\tilde\sigma^2}
- \delta\fP(\Omega_1)}
+ \Order{c\delta^3\tilde\sigma^2}\;,
\end{align}
where $\tilde\sigma$ has been defined in~\eqref{eq:def_sigma_tilde}. 
As mentioned before, we will show in Corollary~\ref{cor:Omega1} that $\fP(\Omega_1)$ 
is exponentially close to $1$. 
The result for the expectation follows from the definition~\eqref{eq:def_c} of $c$ 
and the relation~\eqref{eq:T_delta} between $\delta$ and $T(x_0,y_0)$. The other 
moments are computed in an analogous way. 
\end{proof}


\subsection{Nonlinear equation}
\label{ssec:nonlinear} 

Recall that the moments of $\tau^0$ that we computed correspond to the linearised equation 
\begin{equation}
\label{eq:u0} 
\6u^0_t = (\sigma W_t + 2x^{\det}(t) u^0_t)\6t\;.
\end{equation}
We now extend these results to the non-linear equation
\begin{equation}
\label{eq:ut} 
\6u_t = (\sigma W_t + 2x^{\det}(t) u_t + u_t^2) \6t\;.
\end{equation}
To do that, we show that $z_t=u_t - u_t^0$ remains small with high probability. 
We begin with a preparatory estimate. 

\begin{lemma}
\label{lem:u^0}
There exist constants $M_0, \kappa_0>0$ such that for any 
$H>0$ and $0\leqs t \leqs T_{\max}$, one has 
\begin{equation}
\label{eq:sup_u^0} 
\biggprob{\sup_{0\leqs s\leqs t} \abs{u_s^0} > H}
\leqs M_0Ht \exp\biggset{-\kappa_0\frac{H^2}{\sigma^2t^3}}\;.
\end{equation}
\end{lemma}
\begin{proof}
One can use the bound in~\cite[Theorem~2.2]{BerBle2025} on the supremum of 
a Gaussian process. This requires a bound of the form 
\begin{equation}
\bigexpec{(X_t - X_s)^2} \leqs G \abs{t-s}^\gamma
\end{equation} 
to hold. This is indeed the case with $\gamma = 2$, thanks to the computation of the 
covariance made in~\eqref{eq:rho_stt}.
\end{proof}

This result allows us to show that the probability of the event $\Omega_1$ introduced 
in~\eqref{eq:def_Omega1} is exponentially close to $1$.

\begin{corollary}
\label{cor:Omega1}
There exist constants $M_1, \kappa_1 > 0$ such that 
\begin{equation}
 \fP(\Omega_1^c) \leqs M_1 \frac{\delta^2}{r_0} \e^{-\kappa_1 r_0^3/(\delta\sigma^2)}\;.
\end{equation} 
\end{corollary}
\begin{proof}
It suffices to apply Lemma~\ref{lem:u^0} with $t = \widehat T_{\max}$, which scales like 
$\delta/r_0$, and an $H$ of order $\delta$, chosen in 
such a way that $\abs{u^0_t} \leqs H$ for all $t\leqs \widehat T_{\max}$ 
implies $\tau^0 < \widehat T_{\max}$. 
\end{proof}

We are now able to bound the probability that the difference $z_t = u_t - u^0_t$ 
becomes large. 

\begin{lemma}
\label{lem:sup_z}
For $\delta$ small enough, there exist constants $\kappa, M, h_0>0$ such that for any 
$h\in(0,h_0]$ and $0\leqs t \leqs \widehat T_{\max}$, one has 
\begin{equation} 
\biggprob{\sup_{0\leqs s\leqs t} \abs{z_s} > h}
\leqs M\exp\biggset{-\kappa\frac{h}{\sigma^2t^4}}\;.
\end{equation} 
\end{lemma}
\begin{proof}
Taking the difference of~\eqref{eq:ut} and~\eqref{eq:u0}, 
we find that $z_t$ satisfies
\begin{equation}
\6z_t = \bigbrak{(u^0_t)^2 + 2(x^{\det}_t+u^0_t)z_t + z_t^2} \6t\;.
\end{equation}
We introduce the stopping time 
\begin{equation}
 \tau = \inf\bigsetsuch{s>0}{\abs{z_s}>h}\;.
\end{equation} 
For any decomposition $h = h_1 + h_2 + h_3 + h_4$, we can write 
\begin{equation}
\biggprob{\sup_{0\leqs s\leqs t} \abs{z_s} > h}
\leqs\sum_{j=1}^4\biggprob{\sup_{0\leqs s\leqs t\wedge\tau}\abs{r_j(t)}>h_j}
=: \sum_{j=1}^4 P_j\;,
\end{equation}
where 
\begin{equation}
r_1(t) = \int_0^t (u_s^0)^2\6s\;,  \quad
r_2(t) = 2\int_0^tu_s^0z_s\6s\;,   \quad
r_3(t) = \int_0^tz_s^2\6s\;,   \quad
r_4(t) = 2\int_0^tx^{\det}_sz_s\6s\;.
\end{equation}
Using Lemma~\ref{lem:u^0}, we obtain the existence of constants 
$\kappa_1, \kappa_2, M_1, M_2 > 0$ such that  
\begin{equation}
P_1 \leqs M_1 \exp\biggset{-\kappa_1 \frac{h_1}{\sigma^2t^4}}\;, \qquad 
P_2 \leqs M_2 \exp\biggset{-\kappa_2 \frac{h_2^2}{h^2\sigma^2t^5}}\;.
\end{equation}
Choosing $h_1$ of order $h$ and $h_2$ of order $h^{3/2}t^{1/2}$ yields comparable $P_1$ and $P_2$. 
Taking $h_3 = th^2$ and $h_4$ of order $th$ yields $P_3 = P_4 = 0$. 
Since $t\leqs \widehat T_{\max} = \Order{\delta}$, such a choice of the $h_i$ is possible. 
Combining the bounds yields the result. 
\end{proof}

We have now everything needed to extend the moment estimates in Corollary~\ref{cor:moments_tau0}
to $\tau$. It will be convenient to define the event 
\begin{equation}
\label{eq:def_Omega2} 
 \Omega_2 
 = \biggsetsuch{\omega}{\sup_{0 \leqs t \leqs \widehat T_{\max}} \abs{u_t(\omega)} 
 \leqs a \widehat T_{\max}}\;,
\end{equation} 
where the constant $a$ will be chosen in a suitable way. 

\begin{proposition}
\label{prop:moments_tau}
There exist constants $M, \kappa, a > 0$ 
such that for sufficiently small $\sigma$ and $\delta$, 
\begin{equation}
\label{eq:P_Omega2c} 
 \fP(\Omega_2^c) \leqs M\e^{-\kappa r_0/(\delta\sigma^2)}\;.
\end{equation} 
Furthermore, the moments of $\tau$ satisfy 
\begin{align}
 \bigexpec{\tau\indicator{\Omega_2}}
 &= T(x_0,y_0) + \frac12 \frac{T(x_0,y_0)^2\sigma^2}{r_0^2}
 + \biggOrder{\frac{\delta^3\sigma^2}{r_0^4}}
 + \bigOrder{\e^{-\kappa r_0/\delta}}\;, \\
 \variance\bigbrak{\tau\indicator{\Omega_2}}
 &= \frac13 \frac{T(x_0,y_0)^3\sigma^2}{r_0^2}
 + \biggOrder{\frac{\delta^4\sigma^2}{r_0^5}}
 + \bigOrder{\e^{-\kappa r_0/\delta}}\;, \\
 \bigexpec{\tau^3\indicator{\Omega_2}}
 &= T(x_0,y_0)^3 + \biggOrder{\frac{\delta^4\sigma^2}{r_0^6}}
 + \bigOrder{\e^{-\kappa r_0/\delta}}\;.
\end{align}
\end{proposition}
\begin{proof}
Given $h\in[0,a\widehat T_{\max}]$, we introduce the events 
\begin{align}
\Omega_0(h) 
&= \biggsetsuch{\omega}{\sup_{0\leqs t\leqs \widehat T_{\max}} 
\abs{u^0_t(\omega)} \leqs a \widehat T_{\max} - h}\;, \\
\Omega_3(h) 
&= \biggsetsuch{\omega}{\sup_{0\leqs t\leqs \widehat T_{\max}} 
\abs{u_t(\omega) - u^0_t(\omega)} \leqs h}\;.
\end{align}
It follows from Lemmas~\ref{lem:u^0} and~\ref{lem:sup_z} that 
\begin{align}
 \fP(\Omega_0(h)^c) 
 &\leqs M_0 a\widehat T_{\max}^2
 \exp\biggset{-\kappa_0\frac{(a\widehat T_{\max} - h)^2}{\sigma^2\widehat T_{\max}^3}}\;, \\
\fP(\Omega_3(h)^c) 
&\leqs M_1\exp\biggset{-\kappa_1\frac{h}{\sigma^2\widehat T_{\max}^4}}
\end{align}
for some $M_0, M_1, \kappa_0, \kappa_1 > 0$. 
Since $\Omega_2 \subset \Omega_0(h) \cap \Omega_3(h)$ for any admissible $h$,  
we have $\fP(\Omega_2^c) \leqs \fP(\Omega_0(h)^c) + \fP(\Omega_3(h)^c)$, and the 
bound~\eqref{eq:P_Omega2c} follows by choosing 
\begin{equation}
 h = \widehat T_{\max}^3\;.
\end{equation} 
Regarding the expectations, we observe that if $\omega \in \Omega_2\cap\Omega_3(h)$, 
then $\abs{u^0_t(\omega)} \leqs a\widehat T_{\max} + h$
for all $t\in[0,\widehat T_{\max}]$. If $a$ is a sufficiently small constant
(independent of $\delta$ and $h$), this implies $\tau^0 < \widehat T_{\max}$. 
Therefore, 
\begin{equation}
\label{eq:Omega_123} 
 \Omega_2 \cap \Omega_3(h) \subset \Omega_1\;.
\end{equation} 
For any $h_1\in\R$, we write 
\begin{equation}
\tau^0(h_1) = \inf\setsuch{t>0}{u^0_t = d(t) + h_1}\;.
\end{equation} 
Then by construction, on $\Omega_3(h)$ one has 
\begin{equation}
\tau^0(-h) \leqs \tau \leqs \tau^0(h)\;.
\end{equation} 
Using~\eqref{eq:Omega_123}, this yields the upper bound 
\begin{align}
\expec{\tau\indicator{\Omega_2}} 
&= \expec{\tau\indicator{\Omega_2}\indicator{\Omega_3(h)}} 
+ \expec{\tau\indicator{\Omega_2}\indicator{\Omega_3(h)^c}} \\
&\leqs \expec{\tau^0(h)\indicator{\Omega_2 \cap \Omega_3(h)}} 
+ \widehat T_{\max} \fP(\Omega_3(h)^c)\\
&\leqs \expec{\tau^0(h)\indicator{\Omega_1}} 
+ \widehat T_{\max} \fP(\Omega_3(h)^c)\;,
\end{align}
and the lower bound 
\begin{align}
\expec{\tau\indicator{\Omega_2}} 
&\geqs \expec{\tau^0(-h)\indicator{\Omega_2 \cap \Omega_3(h)}} \\
&= \expec{\tau^0(-h)\indicator{\Omega_1}} 
- \expec{\tau^0(-h)\indicator{\Omega_1 \cap (\Omega_2\cap\Omega_3(h))^c}}\\
&\geqs \expec{\tau^0(-h)\indicator{\Omega_1}}  
- \widehat T_{\max} \bigbrak{\fP(\Omega_2^c) + \fP(\Omega_3(h)^c))}\;.
\end{align}
Corollary~\ref{cor:moments_tau0} implies that for any $h$ of order $\widehat T_{\max}$, 
\begin{equation}
\expec{\tau^0(h)\indicator{\Omega_1}} = T(x_0,y_0) + \frac{h + \Order{\delta h}}{r_0} 
+ \frac12 \frac{T(x_0,y_0)^2\sigma^2}{r_0^2}
+ \biggOrder{\frac{\delta^3\sigma^2}{r_0^4}}\;.
\end{equation} 
The choice 
\begin{equation}
h = \sigma^2 \widehat T_{\max}^3
\end{equation} 
ensures that the term of order $h$ can be incorporated in the error term,
since $\widehat T_{\max}$ scales like $\delta/r_0$. 
Furthermore, with this choice of $h$, $\fP(\Omega_3(h))^c$ has order 
$\e^{-\kappa/\widehat T_{\max}}$. 
This yields the expression for $\expec{\tau\indicator{\Omega_2}}$. 
The other moments are treated in a similar way. 
\end{proof}

\subsection{Moments of the exit location}
\label{ssec:slice_moments} 

Having obtained the moments of $\tau$, we can now proceed to deriving expressions for 
the moments of $y_\tau$, which is the main result of this entire section. 
Here it will be useful to work with the event 
\begin{equation}
\label{eq:def_Omega4} 
 \Omega_4 = \Omega_2 
 \cap \biggsetsuch{\omega}{\sup_{0\leqs t\leqs \widehat T_{\max}} \sigma\abs{W_t(\omega)} 
 \leqs b\widehat T_{\max}}
\end{equation} 
where $b>0$ will be chosen later on. 

\begin{proposition}
\label{prop:moments_ytau} 
There exist constants $M, \kappa > 0$, depending only on $b$, such that 
for sufficiently small $\sigma$ and $\delta$, 
\begin{equation}
 \label{eq:bound_Omega4}
 \fP(\Omega_4^c) \leqs M\e^{-\kappa\delta/(r_0\sigma^2)}\;.
\end{equation} 
Furthermore, the first three moments of $(y_\tau - y_0 - T(x_0,y_0))\indicator{\Omega_4}$ satisfy 
\begin{align}
\bigexpec{\bigpar{y_\tau - y_0 - T(x_0,y_0)}\indicator{\Omega_4}}
&= \frac12 \frac{T(x_0,y_0)^2\sigma^2}{r_0^2}
 + \biggOrder{\frac{\delta^3\sigma^2}{r_0^4}}
 + \bigOrder{\e^{-\kappa r_0/\delta}}\;, 
\label{eq:expect_ytau} 
\\
\bigexpec{\bigpar{y_\tau - y_0 - T(x_0,y_0)}^2\indicator{\Omega_4}}
&= T(x_0,y_0) \sigma^2-\frac{T(x_0,y_0)^2}{r_0}\sigma^2
+\biggOrder{\frac{\delta^3\sigma^2}{r_0^4}}
+ \bigOrder{\e^{-\kappa r_0/\delta}}\;, \\
\bigexpec{\bigpar{y_\tau - y_0 - T(x_0,y_0)}^3\indicator{\Omega_4}}
&= \Order{\sigma^3} + \biggOrder{\frac{\delta^3}{r_0^3}}\;.
\end{align}
\end{proposition}
\begin{proof}
Using either the Bernstein-type estimate~\cite[Lemma~B.1.3]{BGbook},
or~\cite[Theorem~2.2]{BerBle2025}, applied this time to Brownian motion, 
we obtain 
\begin{equation}
 \biggprob{\sup_{0\leqs t\leqs \widehat T_{\max}} \sigma\abs{W_t} 
 > b \widehat T_{\max}}
 \leqs M_0 \e^{-\kappa_0 b^2\widehat T_{\max}^2/(\sigma^2\widehat T_{\max})}
 = M_0 \e^{-\kappa_0 b^2\widehat T_{\max}/\sigma^2}\;.
\end{equation} 
Since $\widehat T_{\max}$ scales like $\delta/r_0$, this yields~\eqref{eq:bound_Omega4}, 
considering the bound~\eqref{eq:P_Omega2c} on $\fP(\Omega_2^c)$. 

Regarding the moments, since $y_t = y_0 + t + \sigma W_t$, we have immediately
\begin{equation}
\bigexpec{y_\tau\indicator{\Omega_2}} 
= y_0\fP(\Omega_2) + \bigexpec{\tau\indicator{\Omega_2}} 
+ \sigma \bigexpec{W_\tau\indicator{\Omega_2}}\;.
\end{equation} 
Since $\tau\wedge\widehat T_{\max}$ is a stopping time, and $(W_t)_{t\geqs0}$ is a martingale, we have 
\begin{align}
 0
 &= \bigexpec{W_{\tau\wedge\widehat T_{\max}}} \\
 &= \bigexpec{W_{\tau\wedge\widehat T_{\max}}\indicator{\Omega_2}}
 + \bigexpec{W_{\tau\wedge\widehat T_{\max}}\indicator{\Omega_2^c}} \\
 &= \bigexpec{W_{\tau}\indicator{\Omega_2}}
 + \bigexpec{W_{\tau\wedge\widehat T_{\max}}\indicator{\Omega_2^c}}\;,
\end{align}
since $\omega\in\Omega_2$ implies $\tau(\omega) \leqs \widehat T_{\max}$. 
The Cauchy--Schwarz inequality implies 
\begin{equation}
 \bigabs{\bigexpec{W_{\tau\wedge\widehat T_{\max}}\indicator{\Omega_2^c}}} 
 \leqs \sqrt{\bigexpec{W_{\tau\wedge\widehat T_{\max}}^2}}
 \sqrt{\fP(\Omega_2^c)}\;.
\end{equation} 
Since $(W_t^2)_{t\geqs0}$ is a submartingale, 
\begin{equation}
\bigexpec{W_{\tau\wedge\widehat T_{\max}}^2}
\leqs \bigexpec{W_{\widehat T_{\max}}^2} 
= \widehat T_{\max} = \Order{c\delta}\;,
\end{equation} 
and the bound~\eqref{eq:expect_ytau} on the expectation, with $\indicator{\Omega_2}$ instead of 
$\indicator{\Omega_4}$, follows from~\eqref{eq:P_Omega2c}. To extend the bound to 
$\indicator{\Omega_4}$, we write
\begin{equation}
 \bigexpec{y_\tau\indicator{\Omega_4}}
 = \bigexpec{y_\tau\indicator{\Omega_2}}
 + \bigexpec{y_\tau\indicator{\Omega_4\cap\Omega_2^c}}\;.
\end{equation} 
Since $y_\tau$ is bounded on $\Omega_4$, the second term on the right-hand side is bounded 
by a constant times $\fP(\Omega_2^c)$, which yields a negligible exponentially small error term.  

To estimate the second moment, we first compute
\begin{equation}
\label{eq:moment2-1} 
\bigexpec{(y_\tau-y_0)^2\indicator{\Omega_2}}
= \bigexpec{\tau^2\indicator{\Omega_2}} 
+ 2\sigma\bigexpec{\tau W_\tau \indicator{\Omega_2}} 
+ \sigma^2\bigexpec{W_\tau^2\indicator{\Omega_2}}\;.
\end{equation}
The term $\expec{\tau^2\indicator{\Omega_2}}$ has been estimated in 
Proposition~\ref{prop:moments_tau}.
We consider next the term $\expec{W_\tau^2\indicator{\Omega_2}}$.
Here we use the fact that the process $(M_t)_{t\geqs0}$ given by 
\begin{equation}
M_t = W_t^2 - t
\end{equation} 
is a martingale. By a similar stopping argument as before, we get 
\begin{equation}
 \bigexpec{W_\tau^2\indicator{\Omega_2}}
 = \bigexpec{\tau\indicator{\Omega_2}}
 + \BigOrder{\sqrt{\bigexpec{M_{\tau\wedge\widehat T_{\max}}^2}}
 \sqrt{\fP(\Omega_2^c)}\,}\;.
\end{equation} 
To compute $\expec{\tau W_\tau \indicator{\Omega_2}}$, we observe that 
since $\6\,(tW_t) = t\6W_t + W_t\6t$, we have the integration-by-parts
relation
\begin{equation}
tW_t = X_t + \int_0^t s\6W_s\;,
\end{equation} 
where $X_t$ is defined in~\eqref{eq:def_Xt}. The stochastic integral being a martingale, 
that we denote $M_t'$, we obtain 
\begin{equation}
\bigexpec{(\tau W_\tau)\indicator{\Omega_2}} 
= \bigexpec{X_\tau\indicator{\Omega_2}}
- \bigexpec{M_{\tau\wedge\widehat T_{\max}}'\indicator{\Omega_2^c}}\;.
\end{equation} 
Here we can reuse the results from Section~\ref{ssec:linearised_moments}, since $\sigma X_t$ 
is equal to $u^0_t$ in the particular case $\alpha(t) = 0$, and $u^0_\tau = d(\tau)$ 
for a suitable boundary $d$. This yields 
\begin{equation}
 \sigma\bigexpec{X_\tau\indicator{\Omega_2}} 
 = T(x_0,y_0)-\bigexpec{\tau\indicator{\Omega_2}}
 = -\frac12 \frac{T(x_0,y_0)^2\sigma^2}{r_0^2}
 + \biggOrder{\frac{\delta^3\sigma^2}{r_0^4}}\;.
\end{equation} 
Combining the different estimates, and using a similar argument to replace 
$\indicator{\Omega_2}$ by $\indicator{\Omega_4}$ 
yields the expression for the second moment. 
The proof of the third moment is similar. 
\end{proof}


\section{Combining the slices}
\label{sec:slices} 


In this section, we provide the proof of Theorem~\ref{thm:main}. 
We fix an initial condition $(\xin,\yin)$ satisfying $\xin^2+\yin > 0$, and 
a final $y$-value $\yf < y^\star$. Let $(x^{\det}(t), y^{\det}(t))$ be 
the deterministic solution~\eqref{eq:xdet} with this initial condition, 
and let $\xf = x^{\det}(\yf)$. 


\subsection{Partition}
\label{ssec:partition} 

\begin{figure}
\begin{center}
\begin{tikzpicture}
[>=stealth',point/.style={circle,inner sep=0.035cm,fill=white,draw},
encircle/.style={circle,inner sep=0.07cm,draw},
x=4.5cm,y=2.5cm,declare function={f(\x) = 0.7*(1 + tanh(\x)) - \x^2/(1+exp(2*\x));}]

\draw[->,semithick] (0,-0.6) -> (0,1.8);
\draw[->,semithick] (-0.7,0) -> (1.5,0);

\draw[blue,thick,-,smooth,domain=-0.7:0.7,samples=75,/pgf/fpu,
/pgf/fpu/output format=fixed] plot ({\x}, {-(\x)^2});

\draw[violet,thick,-,smooth,domain=-0.7:1.4,samples=75,/pgf/fpu,
/pgf/fpu/output format=fixed] plot ({\x}, {f(\x)});


\draw[semithick, blue] (-0.7,1.45) -- (1.4,1.45);
\node[point] at (0.0,1.45) {};
\node[] at (-0.07,1.55) {$y^\star$};

\pgfmathsetmacro{\xx}{0.2}
\draw[semithick, violet] ({\xx},0) -- ({\xx},{f(\xx)}) -- (0,{f(\xx)}); 
\node[point] at ({\xx},0) {};
\node[point] at ({\xx},{f(\xx)}) {};
\node[point] at (0,{f(\xx)}) {};
\node[violet] at ({\xx},-0.15) {$x_{n-1}$};
\node[violet] at (-0.1,{f(\xx)}) {$y^{\det}_{n-1}$};

\pgfmathsetmacro{\yy}{0.53}
\draw[semithick, violet] ({\yy},0) -- ({\yy},{f(\yy)}) -- (0,{f(\yy)}); 
\draw[semithick, dashed, violet] (-0.25,{f(\yy)}) -- (0,{f(\yy)}); 
\node[point] at ({\yy},0) {};
\node[point] at ({\yy},{f(\yy)}) {};
\node[point] at (0,{f(\yy)}) {};
\node[violet] at ({\yy},-0.1) {$x_n$};
\node[violet] at (-0.3,{f(\yy)}) {$y^{\det}_n$};

\pgfmathsetmacro{\zz}{1.0}
\draw[semithick, violet] ({\zz},0) -- ({\zz},{f(\zz)}) -- (0,{f(\zz)}); 
\node[point] at ({\zz},0) {};
\node[point] at ({\zz},{f(\zz)}) {};
\node[point] at (0,{f(\zz)}) {};
\node[violet] at ({\zz},-0.1) {$x_{n+1}$};
\node[violet] at (-0.1,{f(\zz)}) {$y^{\det}_{n+1}$};

\node[violet] at (-0.65,0.5) {$(x^{\det}(t),y^{\det}(t))$};
\node[blue] at (-0.3,-0.3) {$y = -x^2$};

\node[] at (1.4,0.1) {$x$};
\node[] at (0.07,1.65) {$y$};

\end{tikzpicture}
\end{center}
\vspace{-2mm}
\caption{Definition of the partition of $[\xin,\xf]$.
The difference $y^{\det}_{n+1} - y^{\det}_n$ is constant, equal to $\gamma$.
The difference $\delta_n = x_{n+1} - x_n$ is increasing with $n$.}
\label{fig:partition} 
\end{figure}
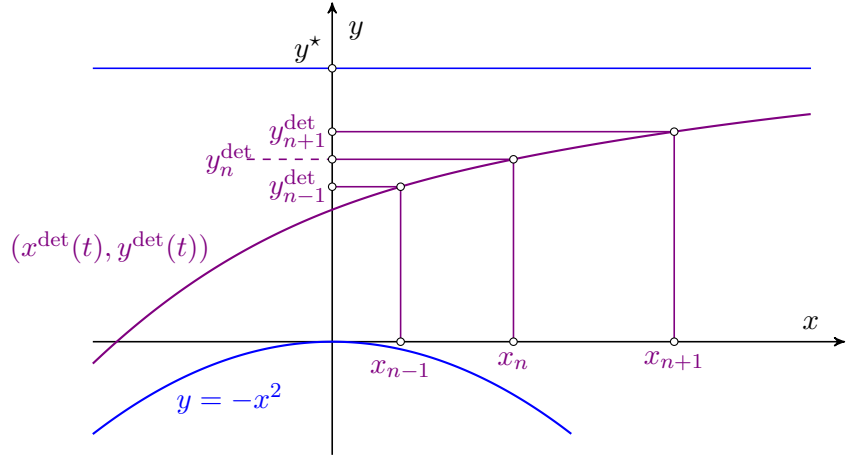

We first choose a partition $(x_n)_{0\leqs n\leqs N}$ of $[\xin,\xf]$ that 
is \lq\lq equidistant in $y$\rq\rq\ (see Figure~\ref{fig:partition}). 
Given an integer $N\geqs1$, that will be taken large, the partition is defined by 
\begin{equation}
 x_n = x^{\det}(n\gamma)\;, 
 \qquad 
 \gamma = \frac{\yf - \yin}{N}\;.
\end{equation}
To lighten notations, we write 
\begin{equation}
 y^{\det}_n = y^{\det}(n\gamma) = \yin + n\gamma\;.
\end{equation} 
We use the superscript ${}^{\det}$ for $y^{\det}_n$ to distinguish it from 
general starting points $y_n$ on the section $\set{x = x_n}$. Since such a 
distinction is not needed for $x_n$, we omit the superscript in that case. 
Note that since 
\begin{equation}
 x_{n+1} = x_n + \gamma(x_n^2 + y^{\det}_n) + \Order{\gamma^2}\;,
\end{equation} 
we have 
\begin{equation}
 \delta_n := x_{n+1} - x_n 
 = \gamma(x_n^2 + y^{\det}_n) + \Order{\gamma^2}\;.
\end{equation} 
Given $y_n > -x_n^2$ and $0 \leqs n < m \leqs N$, we use the notation
\begin{equation}
 T_{n,m}(y_n)
 = T(x_n, y_n; x_m)
\end{equation} 
for the deterministic time needed for the solution starting in $(x_n,y_n)$ to 
reach the line $\set{x = x_m}$ (the notation $T(x_n, y_n; x_m)$ has been introduced 
in~\eqref{eq:def_T}). In particular, we have 
\begin{equation}
 T_{n,m}(y^{\det}_n)
 = y^{\det}_m - y^{\det}_n
 = (m-n)\gamma\;.
\end{equation} 
In other words, the deterministic reference solution $(x^{\det}(t), y^{\det}(t))$ 
takes always the same time $\gamma$ to move from a line $\set{x = x_n}$ to 
the next line $\set{x = x_{n+1}}$. 

In order to work with bounded values of $y_n$, we fix a strictly increasing function 
$h:[\xin,\xf]\to(0,\infty)$, and introduce the increasing 
family of sets 
\begin{equation}
 B_n 
 = \bigsetsuch{(x,y)\in\R^2}{\xin\leqs x\leqs x_n, 
 \abs{y-y^{\det}(x)} \leqs 2h(x)}\;,
\end{equation} 
where $y^{\det}(x)$ stands for the deterministic reference solution parametrized 
by $x$ (the reciprocal of the function $x^{\det}(y)$ defined in~\eqref{eq:xdet_y}). 
We also introduce the intervals 
\begin{equation}
 I_n = [y^{\det}_n - h(x_n), y^{\det}_n + h(x_n)]
\end{equation} 
(see Figure~\ref{fig:Bn}). 
Denoting by $\tau_n$ the first-exit time of the sample path $(x_t,y_t)_{t\geqs0}$ 
from $B_n$, we define the events 
\begin{equation}
 \Omegahat_n 
 = \bigset{x_{\tau_n} = x_n, y_{\tau_n} \in I_n}\;.
\end{equation} 
These events correspond to the sample path remaining between the curves 
$y = y^{\det}(x) - 2h(x)$ and $y = y^{\det}(x) + 2h(x)$, and exiting the set 
$B_n$ through its right-hand boundary, at a point $(x_n, y_{\tau_n})$ such that 
$y_{\tau_n}$ lies in $I_n$. The reason for the factor $2$ in the definition of 
$B_n$ is as follows. Since
\begin{equation}
 T(x_n, y_n; x_{n+1}) = \gamma + \Order{\gamma^2}\;,
\end{equation} 
the deterministic solution starting in $(x_n,y_n)$ with $y_n \in I_n$ will 
cross the next section $\set{x = x_{n+1}}$ at height $y(T)$ satisfying
\begin{equation}
 y(T) = y_n + \gamma + \Order{\gamma^2}\;,
\end{equation} 
so that 
\begin{equation}
 \bigabs{y(T) - y^{\det}_{n+1}}
 \leqs \bigabs{y_n - y^{\det}_n} + \Order{\gamma^2}
 \leqs h(x_n) + \Order{\gamma^2}\;.
\end{equation} 
Since $h$ is strictly increasing, we have $h(x_{n+1}) > h(x_n) + b\gamma$ 
for some $b>0$, and therefore $y(T) \in I_{n+1}$ for $\gamma$ small enough. 
This guarantees that the stochastic sample path starting in $(x_n,y_n)$ will, 
with high probability, reach the line $\set{x = x_{n+1}}$ at a point in $I_{n+1}$.
More precisely, we have the following bound.

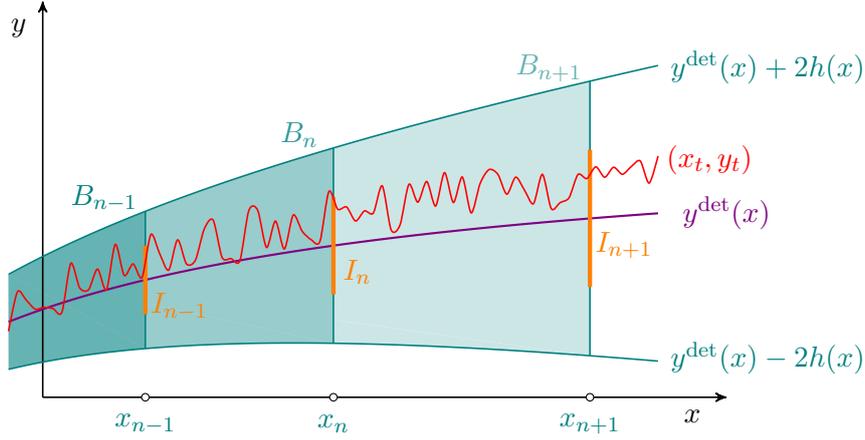
\begin{figure}
\begin{center}
\begin{tikzpicture}
[>=stealth',point/.style={circle,inner sep=0.035cm,fill=white,draw},
encircle/.style={circle,inner sep=0.07cm,draw},
x=4.5cm,y=3.5cm,declare function={y(\x) = 1 - 1/(1.5+\x); h(\x) = 0.2*(1+\x);}]

\pgfmathsetmacro{\xx}{0.3}
\pgfmathsetmacro{\yy}{0.85}
\pgfmathsetmacro{\zz}{1.6}

\path[fill=teal!20,semithick,-,smooth,domain=-0.1:\zz,samples=30,/pgf/fpu,
/pgf/fpu/output format=fixed] 
plot (\x, {y(\x) + h(\x)}) -- (\zz, {y(\zz) - h(\zz)});
\path[fill=teal!20,semithick,-,smooth,domain=\zz:-0.1,samples=30,/pgf/fpu,
/pgf/fpu/output format=fixed] 
plot (\x, {y(\x) - h(\x)}) -- (-0.1, {y(-0.1) + h(-0.1)});

\path[fill=teal!40,semithick,-,smooth,domain=-0.1:\yy,samples=30,/pgf/fpu,
/pgf/fpu/output format=fixed] 
plot (\x, {y(\x) + h(\x)}) -- (\yy, {y(\yy) - h(\yy)});
\path[fill=teal!40,semithick,-,smooth,domain=\yy:-0.1,samples=30,/pgf/fpu,
/pgf/fpu/output format=fixed] 
plot (\x, {y(\x) - h(\x)}) -- (-0.1, {y(-0.1) + h(-0.1)});

\path[fill=teal!60,semithick,-,smooth,domain=-0.1:\xx,samples=30,/pgf/fpu,
/pgf/fpu/output format=fixed] 
plot (\x, {y(\x) + h(\x)}) -- (\xx, {y(\xx) - h(\xx)});
\path[fill=teal!60,semithick,-,smooth,domain=\xx:-0.1,samples=30,/pgf/fpu,
/pgf/fpu/output format=fixed] 
plot (\x, {y(\x) - h(\x)}) -- (-0.1, {y(-0.1) + h(-0.1)});

\draw[->,semithick] (0,0) -> (2,0);
\draw[->,semithick] (0,0) -> (0,1.5);

\draw[violet,thick,-,smooth,domain=-0.1:1.8,samples=75,/pgf/fpu,
/pgf/fpu/output format=fixed] plot ({\x}, {y(\x)});

\draw[teal,semithick,-,smooth,domain=-0.1:1.8,samples=75,/pgf/fpu,
/pgf/fpu/output format=fixed] plot ({\x}, {y(\x) + h(\x)});

\draw[teal,semithick,-,smooth,domain=-0.1:1.8,samples=75,/pgf/fpu,
/pgf/fpu/output format=fixed] plot ({\x}, {y(\x) - h(\x)});


\draw[semithick, teal] ({\xx},{y(\xx) - h(\xx)}) -- ({\xx},{y(\xx) + h(\xx)}); 
\draw[semithick, teal] ({\yy},{y(\yy) - h(\yy)}) -- ({\yy},{y(\yy) + h(\yy)}); 
\draw[semithick, teal] ({\zz},{y(\zz) - h(\zz)}) -- ({\zz},{y(\zz) + h(\zz)}); 


\draw[ultra thick, orange] ({\xx},{y(\xx) - 0.5*h(\xx)}) -- ({\xx},{y(\xx) + 0.5*h(\xx)}); 
\draw[ultra thick, orange] ({\yy},{y(\yy) - 0.5*h(\yy)}) -- ({\yy},{y(\yy) + 0.5*h(\yy)}); 
\draw[ultra thick, orange] ({\zz},{y(\zz) - 0.5*h(\zz)}) -- ({\zz},{y(\zz) + 0.5*h(\zz)});

\draw[red,semithick,-,smooth,domain=-0.1:1.8,samples=74,/pgf/fpu,
/pgf/fpu/output format=fixed] plot ({\x}, {y(\x) + 0.3*h(\x) + 0.1*rand});

\node[point] at ({\xx},0) {};
\node[teal] at ({\xx},-0.1) {$x_{n-1}$};
\node[point] at ({\yy},0) {};
\node[teal] at ({\yy},-0.1) {$x_n$};
\node[point] at ({\zz},0) {};
\node[teal] at ({\zz},-0.1) {$x_{n+1}$};

\node[teal!100] at ({\xx - 0.12}, {y(\xx) + h(\xx) + 0.05}) {$B_{n-1}$};
\node[teal!80] at ({\yy - 0.1}, {y(\yy) + h(\yy) + 0.05}) {$B_n$};
\node[teal!60] at ({\zz - 0.12}, {y(\zz) + h(\zz) + 0.05}) {$B_{n+1}$};

\node[orange] at ({\xx + 0.1}, {y(\xx) - 0.1}) {$I_{n-1}$};
\node[orange] at ({\yy + 0.07}, {y(\yy) - 0.1}) {$I_n$};
\node[orange] at ({\zz + 0.1}, {y(\zz) - 0.1}) {$I_{n+1}$};

\node[violet] at (2.0,0.7) {$y^{\det}(x)$};
\node[teal] at (2.12,1.25) {$y^{\det}(x) + 2h(x)$};
\node[teal] at (2.12,0.15) {$y^{\det}(x) - 2h(x)$};
\node[red] at (1.95,0.9) {$(x_t,y_t)$};
 
\node[] at (1.9,-0.07) {$x$};
\node[] at (-0.07,1.4) {$y$};

\end{tikzpicture}
\end{center}
\vspace{-2mm}
\caption{Definition of the sets $B_n$. Each set extends all the way to 
$x = \xin$, so the sets are nested ($B_{n-1} \subset B_n \subset B_{n+1}$).
The sample path $(x_t,y_t)_t$ is an element of $\Omegahat_{n-1,n+1}$.}
\label{fig:Bn} 
\end{figure}

\begin{proposition}
\label{prop:Omega_n}
There exist constants $\kappa, M > 0$ such that for $\gamma$ small enough and 
any $y_n\in I_n$, one has 
\begin{equation}
\label{eq:bound_P_Omega_n} 
 \fP^{x_n,y_n}(\Omegahat_{n+1}^c)
 \leqs M\e^{-\kappa\gamma/\sigma^2}\;.
\end{equation} 
\end{proposition}
\begin{proof}
Recall that $y_t = y^{\det}_t + \sigma W_t$, and that we have constructed the sets $B_n$ and 
$I_n$ in such a way that the deterministic solution remains at distance $b\gamma$ 
from the boundaries of $B_n$ and $I_{n+1}$ for some $b>0$. Therefore, the set 
$\Omega_4 = \Omega_{4,n}$ introduced in~\eqref{eq:def_Omega4} is included in $\Omegahat_n$, 
for this choice of $b$. Hence, the result follows from~\eqref{eq:bound_Omega4}
and the fact that $\delta/r_0$ scales like $\gamma$. 
\end{proof}

An immediate consequence of this result is that we can bound the probability 
of events 
\begin{equation}
 \Omegahat_{n,m}
 = \bigcap_{p=n+1}^m \Omegahat_p\;, 
 \qquad 
 0 \leqs n < m \leqs N\;,
\end{equation} 
which means that sample paths remain between the curves 
$y = y^{\det}(x) - 2h(x)$ and $y = y^{\det}(x) + 2h(x)$ at least until $x_t$
reaches $x_m$, and cross each line $\set{x = x_p}$ at height $y_{\tau_p} 
\in I_p$. See Figure~\ref{fig:Bn} for an example.

\begin{corollary}
\label{cor:Omega_n}
For any $y_n \in I_n$ and $0 \leqs n < m \leqs N$, one has 
\begin{equation}
 \fP^{x_n,y_n}\bigpar{\Omegahat_{n,m}^c} \leqs (m-n) \e^{-\kappa\gamma/\sigma^2}\;.
\end{equation} 
\end{corollary}
\begin{proof}
The decomposition $\Omegahat_{n,m}^c = \Omegahat_{n+1}^c \cup \Omegahat_{n+1,m}^c$
allows us to write 
\begin{align}
\fP^{x_n,y_n}\bigpar{\Omegahat_{n,m}^c} 
&= \fP^{x_n,y_n}\bigpar{\Omegahat_{n+1}^c \cup \Omegahat_{n+1,m}^c} \\
&= \fP^{x_n,y_n}\Bigpar{\Omegahat_{n+1}^c \cup (\Omegahat_{n+1} \cap \Omegahat_{n+1,m}^c)} \\
&= \fP^{x_n,y_n}\bigpar{\Omegahat_{n+1}^c} + 
\fP^{x_n,y_n}\bigpar{\Omegahat_{n+1} \cap \Omegahat_{n+1,m}^c}\;.
\end{align}
The first term can be bounded by~\eqref{eq:bound_P_Omega_n}, while the strong 
Markov property implies 
\begin{equation}
 \fP^{x_n,y_n}\bigpar{\Omegahat_{n+1} \cap \Omegahat_{n+1,m}^c} 
 \leqs \sup_{y_{n+1}\in I_{n+1}}
 \fP^{x_{n+1},y_{n+1}}\bigpar{\Omegahat_{n+1,m}^c}\;. 
\end{equation} 
The result then follows by induction, starting with $n=m-1$ and moving 
backward in time. 
\end{proof}

Since the event $\Omega_0$ introduced in~\eqref{eq:def_Omega0} contains $\Omegahat_{0,N}$ 
provided $h(x) \leqs h_0$ for all $x$, we have $\fP(\Omega_0^c) \leqs \fP(\Omegahat_{0,N}^c)$, 
and the bound~\eqref{eq:bound_P_Omega0} will follow from our choice of $\gamma$. 


\subsection{Expectation of $y_\tau$}
\label{ssec:expectation} 

Our aim is now to compute the expectations 
\begin{equation}
 E_{n,N}(y_n) = \bigexpecin{(x_n,y_n)}{y_{\tau_N} \indicator{\Omegahat_{n,N}}}
\end{equation} 
for $y_n \in I_n$. Since the starting point $y_n$ may be different from the 
reference solution $y^{\det}_n$, we introduce the notation 
\begin{equation}
 \Pi_n(y_n) = y_n + T_{n,n+1}(y_n)
 \label{eq:Pin} 
\end{equation} 
for its image on $\set{x = x_{n+1}}$ by the deterministic flow. In other words, 
$\Pi_n$ is the deterministic Poincar\'e map from the section $\set{x = x_n}$ 
to the next section $\set{x = x_{n+1}}$. In particular, we have $\Pi_n(y^{\det}_n)
= y^{\det}_n + \gamma = y^{\det}_{n+1}$. 

Then, we have the following inductive relation, which is based on a Taylor 
expansion of $E_{n,N}$. 

\begin{proposition}
\label{prop:induction_E}
Given $1\leqs n \leqs N$ and an initial condition $y_{n-1}\in I_{n-1}$, let $Z_n$ be the random variable 
\begin{equation}
 Z_n = \frac{y_{\tau_n} - \Pi_{n-1}(y_{n-1})}{\sigma}\;.
\end{equation} 
Then one has
\begin{align}
 E_{n-1,N}(y_{n-1})
 ={}&  E_{n,N}\bigpar{\Pi_{n-1}(y_{n-1})} \fP^{x_{n-1},y_{n-1}}(\Omegahat_n) \\
 &{}+ \sigma E_{n,N}'\bigpar{\Pi_{n-1}(y_{n-1})} 
 \bigexpecin{(x_{n-1},y_{n-1})}{Z_n \indicator{\Omegahat_n}} \\
 &{}+ \frac12 \sigma^2 E_{n,N}''\bigpar{\Pi_{n-1}(y_{n-1})} 
 \bigexpecin{(x_{n-1},y_{n-1})}{Z_n^2 \indicator{\Omegahat_n}} \\
 &{}+ \frac16 \sigma^3  
 \bigexpecin{(x_{n-1},y_{n-1})}{Z_n^3 \indicator{\Omegahat_n}
 E_{n,N}'''\bigpar{\Pi_{n-1}(y_{n-1}) + \theta\sigma Z_n}}
 \label{eq:En_Taylor} 
\end{align}
for some $\theta\in[0,1]$. 
\end{proposition}
\begin{proof}
Let $\chi_n$ be the density of $y_{\tau_n}$, restricted to $I_n$ (this can be 
viewed as the density of the process killed if it does not leave $B_n$ through $I_n$). 
Since $\indicator{\Omegahat_{n-1,N}} = \indicator{\Omegahat_{n-1}} \indicator{\Omegahat_{n,N}}$,
the strong Markov property applied on $\set{x = x_n}$ implies 
\begin{align}
 E_{n-1,N}(y_{n-1})
 &= \bigexpecin{(x_{n-1},y_{n-1})}{y_{\tau_N} 
 \indicator{\Omegahat_{n-1}} \indicator{\Omegahat_{n,N}}} \\
 &= \int_{I_n} \chi_n(z) E_{n,N}\bigpar{\Pi(y_{n-1}) + \sigma z} \6z\;.
\end{align}
We use the Taylor expansion with remainder 
\begin{equation}
 E_{n,N}\bigpar{\Pi(y_{n-1}) + \sigma z} 
 = \sum_{k=0}^2 \frac{1}{k!} \sigma^k z^k E_{n,N}^{(k)}\bigpar{\Pi(y_{n-1})}
 + \frac{1}{3!} \sigma^3 z^3 E_{n,N}'''\bigpar{\Pi(y_{n-1}) + \theta\sigma z}\;.
\end{equation} 
Since 
\begin{equation}
 \int_{I_n} z^k \chi_n(z) \6z 
 = \bigprobin{(x_{n-1},y_{n-1})}{Z_n^k \indicator{\Omegahat_n}}
\end{equation} 
for any $k\geqs0$, the first three terms in the Taylor expansion yield the first three terms 
in~\eqref{eq:En_Taylor}. The remainder is dealed with in a similar way, except 
that the function $E_{n,N}'''$ has to remain inside the expectation. 
\end{proof}

From now on, we make the choice 
\begin{equation}
 \gamma = \sigma\;,
\end{equation} 
since this will yield near-optimal error terms, and simplify their expressions. 
Since the event $\Omega_4 = \Omega_{4,n}$ introduced in~\eqref{eq:def_Omega4} 
is included in $\Omegahat_n$, Proposition~\ref{prop:moments_ytau} implies 
\begin{align}
 \bigexpecin{(x_{n-1},y_{n-1})}{Z_n \indicator{\Omegahat_n}} 
&= \frac12 \sigma \frac{T_{n-1,n}(y_{n-1})^2}{(x_n^2 + y^{\det}_n)^2}
+ \Order{\sigma^3} \;, \\
\bigexpecin{(x_{n-1},y_{n-1})}{Z_n^2 \indicator{\Omegahat_n}} 
&= T_{n-1,n}(y_{n-1}) + \Order{\sigma^2}\;. \\
\bigexpecin{(x_{n-1},y_{n-1})}{Z_n^3 \indicator{\Omegahat_n}} 
&= \Order{1}\;.
\end{align}
Here we have replaced the indicator functions $\indicator{\Omega_{4,n}}$ by 
$\indicator{\Omegahat_n}$, since $Z_n$ is bounded on $\Omegahat_n$ and 
the $\Omega_{4,n}$ have an exponentially small probability. 
This allows us to obtain the following inductive relation on the $E_{n,N}(y_n)$.

\begin{proposition}
\label{prop:Dn_induction}
For every $0\leqs n\leqs N-1$, one has 
\begin{equation}
 E_{n,N}(y_n) = y_n + T_{n,N}(y_n) + \sigma^2 D_{n,N}(y_n) 
 + \sigma^3 R_{n,N}(y_n, \sigma) \;,
 \label{eq:EnN_expansion} 
\end{equation} 
where the second-order corrections $D_{n,N}$ and remainders $R_{n,N}$ satisfy 
\begin{align}
\label{eq:DnN_inductive} 
 D_{n-1,N}(y_{n-1})
 &= D_{n,N}\bigpar{\Pi_{n-1}(y_{n-1})} + \frac12\sigma T_{n,N}''\bigpar{\Pi_{n-1}(y_{n-1})}
 + \Order{\sigma^2} \\
 R_{n-1,N}(y_{n-1},\sigma)
 &= R_{n,N}\bigpar{\Pi_{n-1}(y_{n-1}),\sigma} + \Order{\sigma}
\label{eq:RnN_inductive} 
\end{align} 
for $1\leqs n\leqs N-1$. 
\end{proposition}
\begin{proof}
We first note that if $\sigma = 0$, then $\Omegahat_{n,N} = \Omegahat$ and 
\begin{equation}
 E_{n,N}(y_n) 
 = \bigexpecin{(x_n, y_n)}{y_{\tau_N} \indicator{\Omegahat_{n,N}}}
 = y_n + T_{n,N}(y_n)\;,
\end{equation} 
so that~\eqref{eq:EnN_expansion} holds. We now proceed by induction over $n$, going 
backwards in time. The relation~\eqref{eq:DnN_inductive} is true for $n = N$ with 
$T_{N,N}(y_N) = D_{N,N}(y_N) = R_{N,N}(y_N) = 0$. Assuming it holds for some 
$n$ between $1$ and $N$, we have 
\begin{align}
 E_{n,N}'(y) &= 1 + T_{n,N}'(y) + \Order{\sigma^2}\;, \\
 E_{n,N}''(y) &= T_{n,N}''(y) + \Order{\sigma^2}\;.
\label{eq:EnN_derivatives} 
\end{align}
We now plug the expansion~\eqref{eq:En_Taylor} 
into~\eqref{eq:EnN_expansion} and expand into powers of $\sigma$. At order $1$, we 
obtain 
\begin{equation}
 y_{n-1} + T_{n-1,N}(y_{n-1}) 
 = \Pi_{n-1}(y_{n-1}) + T_{n,N}\bigpar{\Pi_{n-1}(y_{n-1})}\;,
\end{equation} 
which is true in view of~\eqref{eq:Pin}, since 
$T_{n-1,N}(y_{n-1}) = T_{n,N}(y_{n-1}) + T_{n,N}(\Pi_{n-1}(y_{n-1}))$ (this 
follows from the semi-group property of the deterministic flow). 
At order $\sigma^2$, we get 
\begin{align}
 D_{n-1,N}(y_{n-1})
 = \lim_{\sigma\to0} \biggbrak{&D_{n,N}\bigpar{\Pi_{n-1}(y_{n-1})}
 + \frac12 E_{n,N}'\bigpar{\Pi_{n-1}(y_{n-1})} 
 \frac{T_{n-1,n}(y_{n-1})^2}{(x_n^2 + y^{\det}_n)^2}\\
 &{}+ \frac12 E_{n,N}''\bigpar{\Pi_{n-1}(y_{n-1})} T_{n-1,n}(y_{n-1})
 + \Order{\sigma^2}}\;.
\end{align}
The bound~\eqref{eq:DnN_inductive} follows from~\eqref{eq:EnN_derivatives} and the fact that 
$T_{n-1,n}(y_{n-1}) = \Order{\sigma}$. 
For the remainder, we obtain the relation 
\begin{align}
 R_{n-1,N}(y_{n-1}, \sigma)
 ={}&  R_{n,N}\bigpar{\Pi_{n-1}(y_{n-1}), \sigma}
 + \frac{\fP^{x_{n-1},y_{n-1}}(\Omegahat_n^c)}{\sigma^3} E_{n,N}\bigpar{\Pi_{n-1}(y_{n-1})} \\
 &{}+ \frac16 \bigexpecin{(x_{n-1},y_{n-1})}{Z_n^3 
 \indicator{\Omegahat_{n,N}}E_{n,N}'''(\Pi_{n-1}(y_{n-1} + \theta\sigma Z_n))} 
 + \Order{\sigma^3}\;,
\end{align} 
which implies the remainder bound~\eqref{eq:RnN_inductive}, since 
$\fP^{x_{n-1},y_{n-1}}(\Omegahat_n^c)$ is exponentially small. 
\end{proof}

Since $\Pi_{n-1}(y^{\det}_{n-1}) = y^{\det}_n$, iterating the inductive 
relation~\eqref{eq:DnN_inductive} shows that 
\begin{equation}
\label{eq:DON} 
 D_{0,N}(y^{\det}_0) 
 = \frac12 \sigma \sum_{n=1}^N T_{n-1,n}''(y^{\det}_n) + \Order{N\sigma^2}\;,
\end{equation} 
and therefore, viewing the sum as a Riemann sum and since $\sigma = \gamma = \Order{1/N}$, 
\begin{align}
 \bigexpecin{(\xin,\yin)}{y_{\tau_N}\indicator{\Omegahat_{0,N}}}
 &= \yf + \frac12\sigma^3 \sum_{n=1}^N T_{n,N}''(y^{\det}_n)
 + \Order{N\sigma^4}
 + \Order{\sigma^3} \\
 &= \yf + \frac12\sigma^2 \int_{\yin}^{\yf} 
 \partial_{yy}T(x^{\det}(y),y;\xf) \6y 
 + \Order{\sigma^3}\;.
\end{align} 
Replacing $\indicator{\Omegahat_{0,N}}$ by $\indicator{\Omega_0}$ only results in an 
error term of order $\fP(\Omegahat_{0,N}^c)$, which is exponentially small. We have thus 
proved the bound~\eqref{eq:bound_expec_main} on the expectation. 


\subsection{Variance of $y_\tau$}
\label{ssec:variance} 

The computation of the variance of $y_\tau$ is quite similar to the computation of the 
expectation, so we only comment on what changes. The quantity
\begin{equation}
 M_{n,N}(y_{n-1}) = \bigexpecin{(x_n,y_n)}{y_{\tau_N^2}\indicator{\Omegahat_{n,N}}}
\end{equation} 
satisfies the expansion~\eqref{eq:En_Taylor} with $E_{n,N}$ replaced by 
$M_{n,N}$. Then the analogue of Proposition~\ref{prop:Dn_induction} reads 
\begin{equation}
 M_{n,N}(y_n) 
 = \bigbrak{y_n + T_{n,N}(y_n)}^2
 + \sigma^2 F_{n,N}(y_n) + \Order{\sigma^3}\;,
\end{equation} 
where
\begin{align}
 F_{n-1,N}(y_{n-1})
 ={}& F_{n,N}\bigpar{\Pi_{n-1}(y_{n-1})} \\
 &{}+ \sigma\Bigbrak{\Bigpar{1+T_{n,N}'\bigpar{\Pi_{n-1}(y_{n-1})}}^2 \\
 &\phantom{{}+{}\sigma\Bigl[} 
 {}+ \Bigpar{\Pi_{n-1}(y_{n-1})+  T_{n,N}\bigpar{\Pi_{n-1}(y_{n-1})}} 
 T_{n,N}''\bigpar{\Pi_{n-1}(y_{n-1})}} \\
 &{}+ \Order{\sigma^2}\;.
\end{align}
Since $y^{\det}_n + T_{n,N}(y^{\det}_n) = y^{\det}_N = \yf$, this entails
\begin{equation}
 F_{0,N}(y^{\det}_0)
 = \sigma \sum_{n=1}^N \Bigbrak{\bigpar{1 + T_{n,N}'(y^{\det}_n)}^2
 + y^{\det}_N T_{n,N}''(y^{\det}_n)}
 + \Order{\sigma^2}\;.
\end{equation} 
Therefore, using~\eqref{eq:DON} one obtains 
\begin{align}
 \variance^{(\xin,\yin)} \bigset{y_{\tau_N} \indicator{\Omegahat_{0,N}}}
 &= \sigma^2 \bigbrak{F_{0,N} - D_{0,N}^2} + \Order{\sigma^3} \\
 &= \sigma^3 \sum_{n=1}^N \bigpar{1 + T'_{n,N}(y^{\det}_n)}^2 + \Order{\sigma^3} \\
 &= \sigma^2 \int_{\yin}^{\yf} 
 \bigpar{1 + \partial_{y}T(x^{\det}(y),y;\xf)}^2 \6y 
 + \Order{\sigma^3}\;.
\end{align}
As in the case of the expectation, replacing $\indicator{\Omegahat_{0,N}}$ by $\indicator{\Omega_0}$ 
only results in an exponentially small error term, so that we have  
proved the bound~\eqref{eq:bound_variance_main} on the variance. The proof of Theorem~\ref{thm:main}
is thus complete.


\section{Analysis of the functions $D$ and $V$}
\label{sec:DV} 

In this section, we analyse the functions 
$D(\xin,\yin;\xf)$ and $V(\xin,\yin;\xf)$ that appear in the statement 
of Theorem~\ref{thm:main}, in the particular case where the deterministic reference 
solution 
\begin{equation}
 (x^{\det}(t), y^{\det}(t))
 =  (x_0^{\det}(t), y_0^{\det}(t))
\end{equation} 
is given by~\eqref{eq:x0}. 
This requires us to find expressions for the derivatives 
$\partial_y T(x_0^{\det}(y),y; \xf)$ and $\partial_{yy} T(x_0^{\det}(y),y; \xf)$
of the deterministic time to reach $\xf$, which we obtain in 
Section~\ref{ssec:T1T2}. Section~\ref{ssec:asympt} is 
dedicated to the asymptotic behaviour as $\xf\to\infty$, providing the 
proof of Proposition~\ref{prop:DV}. 


\subsection{General expressions for the derivatives of $T(\xin,\yin;\xf)$}
\label{ssec:T1T2} 

Since all trajectories in this section are deterministic, we suppress the superscript
$^{\det}$ from now on. Thus, the deterministic reference solution~\eqref{eq:x0} 
is denoted $(x_0(t), y_0(t))$. 

\begin{proposition}
\label{prop:TxTxx}
Assume the initial condition $(\xin, \yin)$ belongs to the deterministic reference 
solution \eqref{eq:x0}. Fix $\xf > \xin$, and denote by $\yf$ the corresponding 
$y$-coordinate. Then the deterministic travel time $T(\xin,\yin; \xf)$ satisfies
\begin{align}
 \partial_y T(\xin,\yin; \xf)
 &= -\frac{A}{\rf}\;, \\
 \partial_{yy} T(\xin,\yin; \xf)
 &= \frac{2\xf\rf-1}{\rf^3} A^2 
 + \frac{2A}{\rf^2} - \frac{B}{\rf}\;,
 \label{eq:TyTyy} 
\end{align}
where 
\begin{alignat}{2}
 A ={}& \rf - \frac{\cin}{\Ai(-\yf)^2}\;,\\
 B ={}& 2\rf\xf + 1 
 +\frac{2}{\Ai(-\yf)^2}
 \biggset{&&- \cin(2\xf - \xin) - \frac12 \Ai(-\yin)^2 \\
 & &&{}+ \pi \cin^2 \biggbrak{\frac{\Bi(-\yin)}{\Ai(-\yin)} - \frac{\Bi(-\yf)}{\Ai(-\yf)}}}\;,
 \label{eq:AB} 
\end{alignat}
and 
\begin{align}
 \rf &= \xf^2 + \yf\;, \\
 \cin &= \Ai(-\yin)^2 (\xin^2 + \yin)
 = \yin \Ai(-\yin)^2 + \Ai'(-\yin)^2\;.
\label{eq:cin} 
\end{align}
\end{proposition}
\begin{proof}
In order to determine the effect of a change in the initial
$y$-coordinate, we write 
\begin{align}
 x(t) &= x_0(t) + \eps x_1(t) + \frac12\eps^2 x_2(t) + \Order{\eps^3} \\
 y(t) &= y_0(t) + \eps\;,
\end{align}
where $x_1(0) = x_2(0) = 0$. Plugging this into the system~\eqref{eq:sn_scaled} 
with $\sigma = 0$ yields 
\begin{align}
 \dot x_1 &= 2x_0(t) x_1 + 1\;, \\
 \dot x_2 &= 2x_0(t) x_2 + 2x_1(t)^2\;.
\end{align}
These are linear inhomogenous equations, which can be solved by the method 
of variation of the constant, yielding
\begin{align}
 x_1(t) 
 &= \frac{I(t)}{\Ai(-\yin-t)^2}\;, \\
 x_2(t)
 &= \frac{2}{\Ai(-\yin-t)^2}
 \int_0^t \frac{I(s)^2}{\Ai(-\yin-s)^2} \6s\;.
 \label{eq:x1x2} 
\end{align}
Here 
\begin{equation}
 I(t) = \int_0^t \Ai(-\yin-s)^2 \6s
 = f(-\yin - t) - f(-\yin)\;,
\end{equation} 
since $f$, which has been introduced in~\eqref{eq:def_f}, satisfies $f'(z) = -\Ai(z)^2$. 
We now expand 
\begin{equation}
 T = T(\xin, \yin + \eps; \xf)
 = T_0 + \eps T_1 + \frac12\eps^2 T_2 + \Order{\eps^3}\;,
\end{equation} 
where $T_0 = T(\xin, \yin; \xf)$, while $T_1$ and $T_2$ coincide with the 
$y$-derivatives we want to determine. The condition $x(T) = \xf$, expanded 
into powers of $\eps$, yields 
\begin{align}
T_1 \dot x_0(T_0) + x_1(T_0) &= 0 \\
T_2 \dot x_0(T_0) + T_1^2 \ddot x_0(T_0) + 2 T_1 \dot x_1(T_0) + x_2(T_0) &= 0\;,
\end{align}
whose solution reads 
\begin{align}
T_1 &= - \frac{x_1(T_0)}{\dot x_0(T_0)} \\
T_2 
&= -\frac{1}{\dot x_0(T_0)} 
\biggbrak{\frac{x_1(T_0)^2 \ddot x_0(T_0)}{\dot x_0(T_0)^2} 
- 2 \frac{x_1(T_0) \dot x_1(T_0)}{\dot x_0(T_0)} + x_2(T_0)}\;.
\label{eq:T1T2} 
\end{align}
Writing 
$A = x_1(T_0)$, $B = x_2(T_0)$, 
and using the differential equations satisfied by $x_0(t)$ and $x_1(t)$ yields 
\begin{align}
 \dot x_0(T_0) &= \xf^2 + \yf = \rf \\
 \ddot x_0(T_0) &= 2\xf^3 + 2 \xf\yf + 1 = 2\xf\rf + 1 \\
 \dot x_1(T_0) &=  2\xf A + 1\;,
\end{align}
which implies the expressions~\eqref{eq:TyTyy} for the derivatives 
$\partial_y T = T_1$ and $\partial_{yy}T = T_2$. It remains to check 
the expressions~\eqref{eq:AB} of $A$ and $B$. 
Regarding $A$, we note that~\eqref{eq:def_f} implies 
\begin{equation}
 f(-y_0(t)) = \Ai(-y_0(t)^2)\bigbrak{y_0(t) + x_0(t)^2}\;.
\end{equation} 
In particular, $f(-\yin) = \cin$, and therefore 
\begin{equation}
\label{eq:It} 
 I(t) = \Ai(-y_0(t))^2 \bigbrak{y_0(t) + x_0(t)^2} - \cin\;.
\end{equation} 
This yields 
\begin{equation}
I(T_0) = \Ai(-\yf)^2 \rf - \cin\;, 
\end{equation} 
which by~\eqref{eq:x1x2} implies the expression for $A$ in~\eqref{eq:AB}. 
Regarding $B$, we replace one factor $I(s)$ in the expression~\eqref{eq:x1x2} 
of $x_2(t)$ by~\eqref{eq:It}, which yields
\begin{equation}
\label{eq:B1} 
 B = \frac{2}{\Ai(-\yf)^2}
 \int_0^{T_0} I(s) \biggbrak{\dot x_0(s) - \frac{\cin}{\Ai(-y_0(s))^2}} \6s\;,
\end{equation} 
and split the integral into two parts. For the first part, integration by 
parts yields 
\begin{align}
 \int_0^{T_0} I(s)\dot x_0(s)\6s 
 &= I(T_0)\xf - \int_0^{T_0} \dot I(s)x_0(s)\6s \\
 &= \bigbrak{\Ai(-\yf)^2\rf - \cin}\xf
 - \int_0^{T_0} \Ai(-y_0(s))\Ai'(-y_0(s))\6s
 \\
 &= \bigbrak{\Ai(-\yf)^2\rf - \cin}\xf
 + \frac12 \bigbrak{\Ai(-\yf)^2 - \Ai(-\yin)^2}\;.
 \label{eq:BI1} 
\end{align}
For the second part, we use the fact that~\eqref{eq:It} implies 
\begin{equation}
 \int_0^{T_0} \frac{I(s)}{\Ai(-y_0(s))^2} \6s 
 = \xf - \xin - \cin \int_0^{T_0} \frac{\6s}{\Ai(-y_0(s))^2}\;, 
 \label{eq:BI2} 
\end{equation} 
where the integral can be computed thanks to the relation 
\begin{equation}
 \pi\frac{\6}{\6z} \biggpar{\frac{\Bi(z)}{\Ai(z)}} = \frac{1}{\Ai(z)^2}\;,
\end{equation} 
which follows from the expression~\eqref{eq:Wronskian} of the Wronskian of Airy functions.
Replacing~\eqref{eq:BI1} and \eqref{eq:BI2} in~\eqref{eq:B1} and rearranging yields the 
expression for $B$ in~\eqref{eq:AB}.
\end{proof}


\subsection{Asymptotics for large $\xf$}
\label{ssec:asympt} 

The expressions for the derivatives $\partial_y T$ and $\partial_{yy}T$ simplify 
somewhat in the limit $\xf\to\infty$, as shows the following result. 

\begin{proposition}
\label{prop:asymptotics}
Assume the initial condition $(\xin, \yin)$ belongs to the deterministic 
solution \eqref{eq:x0}. Then the deterministic travel time $T(\xin,\yin; \xf)$ satisfies
\begin{alignat}{2}
 \lim_{\xf\to\infty}
 \partial_y T(\xin,\yin; \xf)
 ={}& \frac{\cin}{\Ai'(-y^\star)^2} {}-{}&&1\;, 
  \label{eq:Ty_asympt} \\
 \lim_{\xf\to\infty}
 \partial_{yy} T(\xin,\yin; \xf)
 ={}& \frac{2}{\Ai'(-y^\star)^2} 
 \biggset{&&\!\!\!\pi \cin \bigbrak{-\yin\Ai(-\yin)\Bi(-\yin) - \Ai'(-\yin)\Bi'(-\yin)} \\
 & &&\!\!\!{}+ \pi \cin^2 \frac{\Bi'(-y^*)}{\Ai'(-y^*)}
 + \frac12 \Ai(-\yin)^2}\;.
 \label{eq:Tyy_asympt} 
\end{alignat}
\end{proposition}
\begin{proof}
In order to compute the limits, we write 
\begin{equation}
 \yf = y^* - \eta 
\end{equation} 
where $\eta>0$ will be sent to $0$. 
The Taylor expansions 
\begin{align}
 \Ai(-\yf) &= \eta \Ai'(-y^*) \bigbrak{1+\Order{\eta^2}} \;, \\
 \Bi(-\yf) &= \Bi(-y^*) + \eta \Bi'(-y^*) +\Order{\eta^2} \\
 &= -\frac{1}{\pi \Ai'(-y^*)} + \eta \Bi'(-y^*) +\Order{\eta^2}\;,
\label{eq:asymp} 
\end{align}
where we have used the Wronskian identity~\eqref{eq:Wronskian} in the last line, 
immediately imply 
\begin{equation}
 \xf 
 = \frac{\Ai'(-\yf)}{\Ai(-\yf)}
 = \frac{1}{\eta} \bigbrak{1+\Order{\eta^2}}\;, \qquad 
 \rf = \frac{1}{\eta^2} \bigbrak{1+\Order{\eta^2}}\;.
\end{equation} 
It thus follows from~\eqref{eq:AB} that 
\begin{equation}
\label{eq:A_asymp} 
 A = \frac{a_1}{\eta^2} \bigbrak{1+\Order{\eta^2}}
 \qquad 
 \text{where}
 \qquad 
 a_1 = 1 - \frac{\cin}{\Ai'(-y^\star)^2}\;.
\end{equation} 
Replacing this in the expression for $\partial_y T$ in~\eqref{eq:TyTyy} and taking the limit 
$\eta\to0$ yields~\eqref{eq:Ty_asympt}. 
Proceeding similarly for $B$, we find 
\begin{equation}
\label{eq:B_asymp} 
 B = \frac{b_1}{\eta^3} + \frac{b_2}{\eta^2} + \biggOrder{\frac{1}{\eta}}\;,
\end{equation} 
where 
\begin{equation}
 b_1 = 2 + \frac{2}{\Ai'(-y^\star)^2}
 \biggbrak{-2\cin + \frac{\cin^2}{\Ai'(-y^\star)^2}}
 = 2 \biggpar{1- \frac{\cin}{\Ai'(-y^\star)}}^2
 = 2a_1^2
 \label{eq:b1a1} 
\end{equation}
and 
\begin{equation}
 b_2 = \frac{2}{\Ai'(-y^\star)^2} 
 \biggset{\cin\xin - \frac12\Ai'(-\yin)^2 
 + \pi\cin^2 \biggbrak{\frac{\Bi(-\yin)}{\Ai(-\yin)} - \frac{\Bi'(-y^\star)}{\Ai'(-y^\star)}}}\;.
\end{equation} 
This last expression can be simplified by noting that the second form of $\cin$ in~\eqref{eq:cin} 
and the Wronskian identity~\eqref{eq:Wronskian} imply
\begin{equation}
 \xin + \pi\cin \frac{\Bi(-\yin)}{\Ai(-\yin)}
 = \pi\bigbrak{\yin\Ai(-\yin)\Bi(-\yin) + \Ai'(-\yin)\Bi'(-\yin)}\;,
\end{equation} 
which leads to 
\begin{align}
 b_2 = 
 \frac{2}{\Ai'(-y^\star)^2} \biggbrak{&\pi\cin \bigbrak{\yin\Ai(-\yin)\Bi(-\yin) + \Ai'(-\yin)\Bi'(-\yin)}\\
 & -\frac12 \Ai(-\yin)^2 - \pi\cin^2 \frac{\Bi'(-y^\star)}{\Ai'(-y^\star)}} \;.
\end{align}
Replacing~\eqref{eq:A_asymp} and~\eqref{eq:B_asymp} in the second expression in~\eqref{eq:TyTyy} yields 
\begin{equation}
 \partial_{yy}T(\xin,\yin; \xf)
 = \frac{2a_1^2 - b_1}{\eta} - b_2 + \Order{\eta}\;,
\end{equation} 
which yields~\eqref{eq:Tyy_asympt}, by taking the limit $\eta\to0$, since the singular term of order $\eta^{-1}$ 
vanishes thanks to~\eqref{eq:b1a1}.  
\end{proof}

We can now complete the proof of Proposition~\ref{prop:DV}. 

\begin{proof}[{\sc Proof of Proposition~\ref{prop:DV}}]
The limiting values of $D$ and $V$ are given by the integrals 
\begin{align}
 I_a(\yin,y^\star)
 &= \int_{\yin}^{y^\star} \partial_{yy}T(x_0(y),y;\infty) \6y\;, \\
 I_b(\yin,y^\star)
 &= \int_{\yin}^{y^\star} \bigpar{1 + \partial_yT(x_0(y),y;\infty)}^2 \6y\;.
\end{align}
We start by computing $I_b$. Note that $\partial_yT(x_0(y),y;\infty)$ is given by~\eqref{eq:Ty_asympt} 
with $\cin$ replaced by 
\begin{equation}
 c(y) = y\Ai(-y)^2 + \Ai'(-y)^2 
 = f(-y)\;,
\end{equation}
where $f$ has been introduced in~\eqref{eq:def_f}. Therefore, 
\begin{equation}
 I_b(\yin,y^\star) 
 = \frac{1}{\Ai'(-y^\star)^4}
 \int_{-y^\star}^{-\yin} f(z)^2\6z\;,
\end{equation} 
where we have used the change of variables $z = -y$ to avoid sign mistakes when computing the integral. 
The function $\cF$ defined in~\eqref{eq:def_F} satisfies $\cF'(z) = f(z)^2$. Thus, the limiting 
expression for $V$ follows from the fact that $\cF(-y^\star) = -\frac{y^\star}{2}\Ai'(-y^\star)^4$. 
Here the primitive $\cF$ of $f^2$ has been found by using a homogenous polynomial of degree $4$ in
$\Ai(z)$ and $\Ai'(z)$ with polynomial coefficients in $z$ as an ansatz, and determining the 
coefficients by plugging the ansatz into the condition $\cF'(z) = f(z)^2$. 

Proceeding in a similar way for $I_a$, we obtain 
\begin{equation}
 I_a(\yin,y^\star)
 = \frac{1}{\Ai'(-y^\star)^2} \int_{-y^\star}^{-\yin}
 \biggbrak{-2\pi f(z) g(z) + 2\pi \frac{\Bi'(-y^\star)}{\Ai'(-y^\star)} f(z)^2 + \Ai(z)^2} \6z\;, 
\end{equation} 
where 
\begin{equation}
 g(z) = \Ai'(z)\Bi'(z) - z \Ai(z)\Bi(z)\;.
\end{equation} 
The function $\cG(z)$ defined in~\eqref{eq:def_G} is a primitive of $f(z)g(z)$, while we have 
already seen the $\cF(z)$ is a primitive of $f(z)^2$, and $-f(z)$ is a primitive of $\Ai(z)^2$. 
The limiting expression for $D$ then follows from the fact that 
\begin{equation}
 \cG(-y^\star) = -\frac{y^\star}{2} \Ai'(-y^\star)^3 \Bi'(-y^\star) 
 - \frac{1}{8\pi} \Ai'(-y^\star)^2\;.
\end{equation} 
As already stated, the limiting expressions~\eqref{eq:DV_limit} follow from the asymptotical 
behaviour \eqref{eq:Airy_asymptotics} of the Airy functions. This concludes the proof of 
Proposition~\ref{prop:DV}, except for the fact that $D$ is positive, which we show in 
the separate Lemma below.
\end{proof}

\begin{lemma}
We have
\begin{equation}
 D(\xin,\yin;\infty) :=
 \lim_{\xf\to\infty} D(\xin,\yin;\xf) > 0
\end{equation} 
for any $(\xin,\yin)$ on the reference slow solution. 
\end{lemma}
\begin{proof}
Since the limit as $\yin \to -\infty$ is equal to $\frac34$, and the integral 
defining $D$ vanishes when $\yin = y^\star$, it suffices to show that $D$ is decreasing 
in $\yin$. 
The derivative of $D$ with respect to $\yin$ is given by 
$-F(-\yin)/\Ai'(-y^\star)^2$, where 
\begin{equation}
F(z) :=2\pi\Biggpar{\frac{\Bi'(-\ystar)}{\Ai'(-\ystar)}f(z) - g(z)}f(z) 
+ \Ai(z)^2\;.
\end{equation}
Note in particular that 
\begin{equation}
F(-\ystar)=0\;, 
\qquad 
\lim_{z\to\infty}F(z)=0\;,
\end{equation}
A simple way to study the sign of $F$ in the interval $(-y^\star,\infty)$ 
is to study the behaviour of its derivative. One finds, after a computation, 
\begin{equation}
F'(z)=4\pi\Ai(z)^2 G(z)\;, 
\qquad
G(z):= g(z)-\frac{\Bi'(-\ystar)}{\Ai'(-\ystar)}f(z)\;.
\end{equation}
Hence the sign of $F'$ is the same as the sign of $G$ for $z>-\ystar$. 
Furthermore, 
\begin{equation}
G(-\ystar)=0\;, 
\qquad 
\lim_{z\to\infty}G(z)=0\;,
\end{equation}
Again, to study the sign of $G$, we compute its derivative, which is  
\begin{equation}
G'(z) =\Ai(z)H(z)\;, \qquad H(z):= \frac{\Bi'(-\ystar)}{\Ai'(-\ystar)}\Ai(z)-\Bi(z)\;.
\end{equation}
Since $\Ai(z)>0$ for all $z>-\ystar$, $G'$ has the same sign as $H$. 
Thus the proof reduces to the study of $H$, which is much simpler than $F$ and $G$. 
One checks that 
\begin{equation}
H(-\ystar)=-\Bi(-\ystar)=\frac{1}{\pi\Ai'(-\ystar)}>0\;, 
\qquad 
\lim_{z\to\infty} H(z)=-\infty\;,
\end{equation}
and direct differentiation gives 
\begin{equation}
H'(-\ystar)=0\;.
\end{equation}
This means that $H$ starts positive at $z=-\ystar$, with horizontal tangent there, and 
tends to $-\infty$ as $z\to \infty$. We want to prove that $H'(z) < 0$ for all $z > -y^\star$. 
Indeed, this would imply that $G$, and thus $F$, is first increasing, and then decreasing to 
zero, and therefore strictly positive on $(-y^\star, \infty)$.

We first prove that $H'$ is strictly negative on $(-\ystar, 0)$. 
Differentiating $H$ gives after some rearranging
\begin{equation}
H'(z)=\Bi'(z)\frac{\Bi'(-\ystar)}{\Ai'(-\ystar)}(k(z)-k(-\ystar))\;, 
\qquad 
k(z):=\frac{\Ai'(z)}{\Bi'(z)}\;,
\end{equation}
Using the Wronskian identity~\eqref{eq:Wronskian}, one gets
\begin{equation}
k'(z)=\frac{z}{\pi \Bi'(z)^2}\;,
\end{equation}
which is negative for $-\ystar<z<0$. 
It is known that there exists a value $-y^\star_1 \in (-y^\star,0)$ such that 
$\Bi'(z) < 0$ on $[-y^\star,-y^\star_1)$ and $\Bi'(z) > 0$ on $(-y^\star_1,\infty)$. 
It follows that $k(z)$ is negative and decreasing on $(-y^\star,-y^\star_1)$, 
with a vertical asymptote at $-y^\star_1$, and positive and decreasing on 
$(-y^\star_1,\infty)$. In both cases, $H'(z)$ is strictly negative. 

It remains to show that $H'(z)$ is strictly negative on $[0,\infty)$. 
Since $\Ai$ and $\Bi$ are independent solutions of the equation $y''=zy$, 
any linear combination satisfies the same ODE. Therefore
\begin{equation}
\label{eq:ODE}
H''(z) = zH(z)\;.
\end{equation}
This ODE leaves the sector ${H<0, H'<0}$ invariant for $z>0$. 
One can check, using properties of Airy functions, that $H(0)<0$
and $H'(0) < 0$. Therefore, the proof is complete. 
\end{proof}

 
\appendix 
 
\section{Airy functions}
\label{sec:airy} 


The Airy functions $\Ai(z)$ and $\Bi(z)$ are by definition two independent 
solutions of the linear, non-autonomous ordinary differential equation 
\begin{equation}
 y''(z) = z y(z)\;.
\end{equation} 
They can be defined by 
\begin{align}
\Ai(z) &= \frac1\pi \int_0^\infty \cos\biggpar{\frac{\xi^3}{3} + z\xi} \6\xi \\
\Bi(z) &= \frac1\pi \int_0^\infty 
\biggbrak{\exp\biggpar{-\frac{\xi^3}{3} + z\xi} + \sin\biggpar{\frac{\xi^3}{3} + z\xi}} \6\xi\;.
\end{align} 
See Figure~\ref{fig:Airy} for plots.

\begin{figure}[h!]
\begin{center}
\input{TIKZ/figAiry.tex}
\end{center}
\caption[]{Airy functions $\Ai(z)$ and $\Bi(z)$.}
\label{fig:Airy} 
\end{figure}

Their Wronskian satisfies 
\begin{equation}
\label{eq:Wronskian} 
W(z) 
:= \det 
\begin{pmatrix}
\Ai(z) & \Bi(z) \\
\Ai'(z) & \Bi'(z)
\end{pmatrix}
= \Ai(z)\Bi'(z) - \Ai'(z)\Bi(z)
= \frac{1}{\pi}\;.
\end{equation} 
Their asymptotic behaviour for large positive $z$ (see for instance~\cite{Olver_book}) is given by 
\begin{align}
\Ai(z) &= \frac{z^{-1/4}}{2\sqrt{\pi}} \e^{-\xi} u(-\xi)\;, &
\Bi(z) &= \frac{z^{-1/4}}{\sqrt{\pi}} \e^{\xi} u(\xi)\;, \\ 
\Ai'(z) &= -\frac{z^{1/4}}{2\sqrt{\pi}} \e^{-\xi} v(-\xi)\;, &
\Bi'(z) &= \frac{z^{1/4}}{\sqrt{\pi}} \e^{\xi} v(\xi)\;,  
\label{eq:Airy_asymptotics} 
\end{align}
where $\xi = \frac23 z^{2/3}$ and $u$ and $v$ have asymptotic expansions 
starting with 
\begin{equation}
u(\xi) = 1 + \frac{5}{72\xi} + \biggOrder{\frac{1}{\xi^2}}\;, \qquad 
v(\xi) = 1 - \frac{7}{72\xi} + \biggOrder{\frac{1}{\xi^2}}\;.
\end{equation}

\bibliographystyle{plain}
{\small \bibliography{BBZ}}

\newpage

\tableofcontents

\vfill

\bigskip\bigskip\noindent
{\small
Baptiste Bergeot \\
INSA CVL, Univ Orléans, Univ Tours, LaMé, UR 7494, Blois, France\\
{\it E-mail address: }
{\tt baptiste.bergeot@insa-cvl.fr}
}

\bigskip\bigskip\noindent
{\small
Nils Berglund \\
Institut Denis Poisson (IDP) \\ 
Universit\'e d'Orl\'eans, Universit\'e de Tours, CNRS -- UMR 7013 \\
B\^atiment de Math\'ematiques, B.P. 6759\\
45067~Orl\'eans Cedex 2, France \\
{\it E-mail address: }
{\tt nils.berglund@univ-orleans.fr}

\bigskip\bigskip\noindent
{\small
Israa Zogheib \\
Institut Denis Poisson (IDP) \\ 
Universit\'e d'Orl\'eans, Universit\'e de Tours, CNRS -- UMR 7013 \\
B\^atiment de Math\'ematiques, B.P. 6759\\
45067~Orl\'eans Cedex 2, France \\
{\it E-mail address: }
{\tt israa.zogheib@etu.univ-orleans.fr}

\end{document}